\documentclass[12pt,a4paper,reqno]{amsart}  
\title{
  Braided Hopf algebras and  gauge transformations II:  \\ ~ \\
   $*$-structures and examples}

\date{November 2022}
\author{Paolo Aschieri, Giovanni Landi, Chiara Pagani}

\usepackage{amsmath, amssymb, amsthm,bbm} 
\usepackage[english]{babel} 
\usepackage{fullpage}
\usepackage{upgreek}  
\usepackage{enumerate}
\usepackage[all]{xy}
\usepackage{graphicx}
\usepackage{mathtools}
\usepackage{mathrsfs} 
\usepackage{hyperref}
\usepackage{setspace}
\usepackage{graphicx}
\usepackage{float}
\usepackage{dsfont} 
\usepackage{caption}

\numberwithin{equation}{section}

\let\oldtocsection=\tocsection
\let\oldtocsubsection=\tocsubsection
\let\oldtocsubsubsection=\tocsubsubsection
 \renewcommand{\tocsection}[2]{\hspace{0em}\oldtocsection{#1}{#2}}
\renewcommand{\tocsubsection}[2]{\hspace{1em}\oldtocsubsection{#1}{#2}}
\renewcommand{\tocsubsubsection}[2]{\hspace{2em}\oldtocsubsubsection{#1}{#2}}

\address[]{\textit{Paolo Aschieri}  \newline \indent 
Universit{\`a} del Piemonte Orientale, 
\newline \indent 
Dipartimento di Scienze e Innovazione Tecnologica
\newline \indent   viale T.~Michel~11,~15121~Alessandria,~Italy,
\newline \indent  and INFN Torino, via P.~Giuria~1, 10125~Torino,~Italy}
\email{paolo.aschieri@uniupo.it}

\address[]{\textit{Giovanni Landi }
\newline \indent  Universit\`a di Trieste, Via A. Valerio 12/1, 34127 Trieste, Italy,
\newline \indent Institute for Geometry and Physics (IGAP) Trieste, Italy, 
\newline \indent and INFN, Sezione di Trieste, Trieste, Italy. }
\email{landi@units.it}

\address[]{\textit{Chiara Pagani} 
\newline \indent     Universit\`a di Trieste, Via A. Valerio, 12/1, 34127  Trieste, Italy. }
\email{cpagani@units.it}

\theoremstyle{plain}
\newtheorem{thm}{Theorem}[section]
\newtheorem{lem}[thm]{Lemma}
\newtheorem{prop}[thm]{Proposition}

\newtheorem{cor}[thm]{Corollary}

\theoremstyle{definition}
\newtheorem{defi}[thm]{Definition}

\theoremstyle{remark}
\newtheorem{ex}[thm]{Example}
\newtheorem{Q}{Question/Comment}
\newtheorem{rem}[thm]{Remark}

\newcommand{\II}{\mathds{1}}
\newcommand{\dd}{\mathcal{D}}

\newcommand{\nn}{\nonumber}
\newcommand{\ot}{\otimes}
\newcommand{\beq}{\begin{equation}}
\newcommand{\eeq}{\end{equation}}

\newcommand{\btrl}{\blacktriangleright}
\newcommand{\trl}{\vartriangleright}

\newcommand{\bbK}{\mathds{k}}
\newcommand{\id}{\mathrm{id}}

\newcommand{\M}{\mathcal{M}}
\renewcommand{\O}{\mathcal{O}}

\newcommand{\IR}{\mathbb{R}}

\newcommand{\IN}{\mathbb{N}}

\newcommand{\IT}{\mathbb{T}}
\newcommand{\IZ}{\mathbb{Z}}

\newcommand{\can}{\chi}

\newcommand{\flip}{\tau}

\renewcommand{\u}{\textsf{u}}

\newcommand{\uf}{\mathsf{u}_\F}
\newcommand{\buf}{\bar{\mathsf{u}}_\F}
\renewcommand{\r}{\mathsf{R}}
\newcommand{\br}{{\overline{\r}}}

\newcommand{\rF}{{\r_\F}}
\newcommand{\brF}{{\br_\F}}

\newcommand{\KF}{{K_\F}}

\newcommand{\cun}{\varepsilon}

\renewcommand{\cot}{\gamma}
\newcommand{\co}[2]{\cot\left({#1}\ot{#2}\right)}
\newcommand{\coin}[2]{\bar\cot\left({#1}\ot{#2}\right)}

\newcommand{\F}{\mathsf{F}}

\newcommand{\bF}{{\overline{\mathsf{F}}}}
\newcommand{\otF}{\ot_\F}

\newcommand{\dotF}{{\scriptstyle{\!\:\bullet}}_{\scriptscriptstyle{\F\!\:}}}
\newcommand{\dott}{{\scriptstyle{\bullet}_\theta}}

\newcommand{\zero}[1]{{#1}_{\scriptscriptstyle{(0)}}}
\newcommand{\one}[1]{{#1}_{\scriptscriptstyle{(1)}}}
\newcommand{\two}[1]{{#1}_{\scriptscriptstyle{(2)}}}

\newcommand{\fone}[1]{{#1}_{\scriptscriptstyle{[1]}}}
\newcommand{\ftwo}[1]{{#1}_{\scriptscriptstyle{[2]}}}

\newcommand{\aut}[1]{\mathrm{aut}^\r_B(#1)}

\newcommand{\autF}[1]{\mathrm{aut}^{\r_F}_{B_\F}(#1)}

\newcommand{\Der}[1]{\mathrm{Der}{(#1)}}

\newcommand{\DerH}[1]{\mathrm{Der}^\r_{{\M^H}}{(#1)}}
\newcommand{\DerM}[1]{\mathrm{Der}_{{\M^H}}{(#1)}}
\newcommand{\DerHF}[1]{\mathrm{Der}^{\r_\F}_{{\M^{H_\F}}}{(#1)}}

\newcommand{\g}{\mathfrak{g}}
\newcommand{\Hom}{{\rm{Hom}}}

\newcommand{\U}{\mathcal{U}}

\newcommand{\stwo}{\tfrac{1}{\sqrt{2}}}
\newcommand{\parn}[1]{\partial_{\scriptscriptstyle{#1}}}

\newcommand{\tpr}{\,{\raisebox{.28ex}{\rule{.6ex}{.6ex}}}\,}
\renewcommand{\bot}{\boxtimes}

\newcommand{\wee}{\mathsf{K}}
\newcommand{\el}{\mathcal{E}}

\newcommand{\defp}{\lambda}
 
\renewcommand{\bot}{\boxtimes}

\newcommand{\wdg}[2]{\, \textsf{#1}  \wedge \textsf{#2}}

\newcommand{\alh}{\upalpha}
\newcommand{\beh}{\upbeta}
\newcommand{\xh}{\mathsf{x}}
\newcommand{\phia}{\:\!\varphi_{\!\scriptscriptstyle{\alh}}}
\newcommand{\phib}{\:\!\varphi_{\!\scriptscriptstyle{\beh}}}

\newcommand{\suc}{\succ}

\begin{document}

\begin{abstract} 
We consider noncommutative principal bundles which are equivariant under a triangular 
Hopf algebra. We present explicit examples of infinite dimensional
braided Lie and Hopf algebras of infinitesimal gauge transformations
of bundles on noncommutative spheres.
The braiding of these algebras is implemented by  the triangular structure of the symmetry Hopf algebra.
We present a systematic analysis of compatible  $*$-structures,
encompassing the quasitriangular case.
\end{abstract}

\maketitle

\begin{spacing}{0.01}
  \tableofcontents
\end{spacing}

\parskip =.75 ex
\allowdisplaybreaks[4]

\section[Intro]{Introduction}

The gauge group of a principal bundle can be given as bundle automorphisms (diffeomorphisms of the total space onto itself which respect the group action) covering the identity map on the base space. 
Elements of the gauge group act by pullback on the space of connection one-forms on the bundle, thus playing a central role for the definition of the moduli space of connections. 
In the dual algebraic language a principal bundle is given as an algebra extension $B\subseteq A $ which is $H$-Hopf--Galois, for $H$ a Hopf algebra.  A gauge transformation would then be given as an $H$-equivariant algebra morphism of $A$ onto itself which restricts to the identity on $B$. 
This dual definition works well for the case of commutative algebras, and also for 
Hopf--Galois extensions with $H$ coquasitriangular and $B$ commutative \cite{pgc}. It is however too restrictive in general, due to the scarcity of morphisms for a generic noncommutative algebra.
Finding a good notion of bundle automorphisms and gauge transformations for noncommutative principal bundles is still an open problem. 

In our previous paper \cite{pgc-main} we looked at  the problem from the infinitesimal view-point by considering 
algebra derivations, rather than algebra morphisms. We then considered the case of 
$H$-Hopf--Galois extensions
which are equivariant under a triangular Hopf algebra $(K,\r)$. 
 Infinitesimal gauge transformations are now given by $H$-comodule maps that are vertical braided derivations, with the brading implemented by the triangular structure $\r$ of the symmetry $K$. These maps form a braided Lie algebra
$\aut{A}$  and lead to a braided Hopf algebra $\mathcal{U}(\aut{A})$
 of infinitesimal gauge transformations.
The construction is  shown to be compatible with the theory of Drinfeld twists, and thus  suitable for the study of
noncommutative principal bundles that are obtained via twist deformation
(quantization) of classical ones.
We  refer to  \cite{pgc-main} for details and for a discussion of different approaches and of 
the literature on the subject.

 In the present paper we illustrate the general theory, developed in \cite{pgc-main},  
with the computation of the braided Lie algebras of infinitesimal gauge transformations of two important examples of noncommutative principal bundles. These  are given by two Hopf--Galois extensions of  the algebra $\O(S^4_\theta)$ of the noncommutative $4$-sphere $S^4_\theta$ of \cite{cl01}. 
After a  brief review in \S \ref{se:hge} of the general theory, in \S \ref{sec:inst} we determine the braided Lie algebra  $\mathrm{aut}_{\O(S^4_\theta)}(\O(S^7_\theta))$ of infinitesimal gauge transformations of the 
 $\O(SU(2))$-Hopf--Galois extension $\O(S^4_\theta) \subset \O(S^7_\theta)$ of \cite{LS0}.
This bundle can  also be
 obtained as a deformation by a twist on $\O(\IT^2)$ of the Hopf--Galois extension
 $\O(S^4) \subset \O(S^7)$ of the  classical
 $SU(2)$-Hopf bundle \cite{ppca}.  This allows for the construction of 
 $\mathrm{aut}_{\O(S^4_\theta) }(\O(S^7_\theta) )$ from its classical
 counterpart, following  the general theory. The explicit
 description of the
 classical gauge Lie algebra
 $\mathrm{aut}_{\O(S^4)}(\O(S^7))$  relies on the $Spin(5)$
 equivariance of the principal bundle $S^7\to S^4$. This equivariance  also 
implies that, as linear space, 
$\mathrm{aut}_{\O(S^4)}(\O(S^7))  $ 
splits as a direct sum over a class of representations of the Lie algebra $so(5)$, of vertical $\O(SU(2))$-equivariant derivations, see \S \ref{se:rtd}.
Following a similar procedure,  in  \S \ref{S2n} we compute the braided Lie algebra  of infinitesimal gauge transformations of the  $\O(SO_\theta(4,\mathbb{R}))$-Hopf--Galois extension  
$\O(S^{4}_\theta)  \subset
\O(SO_\theta(5,\mathbb{R}))$ of the quantum homogeneous space  $\O(S^{4}_\theta)$ \cite{var,ppca}.    
 Among the tools developped in the paper we mention a systematic analysis in \S\ref{sec:*} of  $*$-structures on 
 braided Hopf algebras associated with quasitriangular Hopf
 algebras and their compatibility with actions on
 $*$-algebras. In the triangular case we further consider braided Lie $*$-algebras and
 their representations on $*$-algebras. 
  
The braided Lie algebras of gauge transformations of these   Hopf--Galois extensions of  $\O(S^4_\theta)$ are re-obtained in \cite{pgc-Atiyah} via an intrinsic construction, which does not  
use the twist procedure. They are further studied there in the context of Atiyah sequences of braided Lie algebras, generalising the Atiyah sequence of a principal bundle.

\section{Braided Lie algebras of gauge transformations}\label{se:hge}

The main objects investigated in this paper are $K$-equivariant Hopf--Galois extensions, for $(K, \r)$ a triangular Hopf algebra, and their braided Lie algebras of gauge symmetries. We briefly 
recall from \cite{pgc-main} the main notions and results that are needed. 

We work in the category of $\bbK$-modules for $\bbK$ a commutative
field, or the ring of formal power series in an indeterminate and
coefficients in a field.
All algebras  are  assumed to be unital and associative; morphisms of
algebras preserve the unit. 
Dually for coalgebras. We use standard terminologies and notations in
Hopf algebra theory. 
For
$H$ a bialgebra we also call  $H$-equivariant a map of $H$-modules or $H$-comodules.

Recall that an algebra $A$ is a right $H$-comodule algebra for a Hopf algebra $H$ 
if it carries a right coaction $\delta: A \to A \ot H$ which is a morphism of algebras.  As usual we write $\delta(a)= \zero{a} \ot \one{a}$ in Sweedler notation with an implicit sum.   Then the subspace of coinvariants $B:= A^{coH}=\big\{b\in A ~|~ \delta (b) = b \otimes 1_H \big\}$ is a subalgebra of $A$.  
The algebra extension $B\subseteq 
A$ is called  an $H$-Hopf--Galois extension if the canonical  map
\begin{align}\label{canonical}  \can := (m \ot \id) \circ (\id \otimes_B \delta ) : A \otimes_B A  \longrightarrow A \ot H~  ,  \quad a' \ot_B a \longmapsto a' a_{\;(0)} \ot a_{\;(1)} 
\end{align} 
is bijective. There may be additional requirements, such as faithful flatness of $A$ as a right $B$-module, 
to be mentioned when needed.

In the present paper we deal with $H$-Hopf--Galois extensions  which are
$K$-equivariant for a Hopf algebra $K$.
That is $A$ carries also a left action $\trl \, : K\otimes A \to A$ that commutes with the right $H$-coaction,
$ \delta\!\:\circ
\trl= (\trl \ot \id) \circ (\id\otimes \delta)$ (the coaction $\delta$ is a $K$-module map where $H$ has trivial $K$-action).
On elements  $k\in K, a\in A$,
\beq\label{compatib}
\zero{(k \trl a)} \ot \one{(k \trl a)} = ( k \trl \zero{a} ) \ot \one{a} \; .
\eeq 
We further assume the Hopf algebra $K$ to be quasitriangular. Recall that a bialgebra (or Hopf algebra) $K$ is quasitriangular if there exists an invertible element  $\r \in K \ot K$ (the universal $R$-matrix of $K$) with respect to which the coproduct
  $\Delta$ of $K$ is quasi-cocommutative
\beq\label{iiR} 
\Delta^{cop} (k) = \r \Delta(k) \br
\eeq for each $k \in K$, with $\Delta^{cop} := \flip \circ \Delta$, $\tau$ the flip map, 
 and $\br \in K \ot K$ the inverse of $\r$, $\r \br = \br \r = 1 \ot 1 $.
Moreover $\r$ is required to satisfy, 
\beq \label{iii} (\Delta \ot \id) \r=\r_{13} \r_{23} \qquad \mbox{and} \qquad 
(\id \ot \Delta)\r=\r_{13} \r_{12}.
 \eeq
 We write $\r:=\r^\alpha \ot \r_\alpha $ with an implicit sum. Then
$\r_{12}=\r^\alpha \ot \r_\alpha \ot 1$,  and similarly for $\r_{23}$ and $\r_{13}$. 
From conditions \eqref{iiR} and \eqref{iii} it follows that $\r$ satisfies the quantum Yang--Baxter equation
$
\r_{12} \r_{13} \r_{23}= \r_{23} \r_{13}\r_{12} 
$. 
The $\r$-matrix  of a quasitriangular bialgebra $(K,\r)$ is unital:
$(\cun \ot \id) \r = 1 =  (\id \ot \cun) \r$, 
with $\cun$ the counit of $K$.
When $K$ is a  Hopf algebra, the quasitriangolarity implies  that its antipode $S$ is invertible and satisfies
\beq\label{ant-R}
(S \ot \id) (\r)= \br  \; ; \quad
(\id \ot S) (\br)= \r \; ; \quad
(S \ot S) (\r)= \r \; .
\eeq
The Hopf algebra $K$ is said to be triangular when 
$\br= {\r}_{21},$ with ${\r}_{21}= \flip(\r)=\r_\alpha \ot \r^\alpha$.

\subsection{Braided Hopf algebras} \label{se:bha}
We recall that a braided bialgebra associated with a quasitriangular
Hopf algebra $( K,\r)$ is 
   a $K$-module  $(L ,\trl_L)$ which is both a $ K$-module algebra 
$(L, m_L, \eta_L, \trl_L)$ and a $ K$-module coalgebra $(L, \Delta_L,\cun_L, \trl_L)$ and is  
a  bialgebra
in the braided monoidal category of $ K$-modules.
That is, $\cun_L : L \to \bbK$ and  $\Delta_L : L \to L  \bot L $ are algebra maps 
with respect to the product   in $L$ and 
the product $\tpr$  in $L\ot L$   defined by
\beq\label{tpr}
(  x \ot   y) \tpr (  x' \ot   y'):= x  \Psi_\r( x' \ot y )  y' =
  x(\r_\alpha \trl_L   x') \ot (\r^\alpha \trl_L   y)   y'\, 
\eeq
for $ x, y, x', y' \in L$ 
and 
\beq\label{braiding}
\Psi_\r:L\ot L\to L\ot L~,\qquad \Psi_\r(x\ot
  y)=\r_\alpha \trl_L y\ot \r^\alpha\trl_L x
\eeq
the braiding. 
We denote $L \bot L = (L \ot L, \tpr)$. 
It is a $K$-module algebra with action
\beq\label{probot}
k \trl_{L \bot L}  (x \bot y) := (\one{k} \trl _L x) \bot (\two{k} \trl _L y).
\eeq

Such an   $L$ is a braided Hopf algebra if there is a $ K$-module map
$S_L : L\to L $, the braided antipode, which satisfies
\beq 
m_L\circ (\id_L\otimes S_L)\circ \Delta_L=
\eta_L\circ \cun_L =
m_L\circ (S_L\otimes\id_L)\circ \Delta_L ~. \label{bantipode}
\eeq
It turns out that $S_L$ is a  braided algebra map
\beq\label{S-algmapA}
S_L (xy)  = ( \r_\alpha \trl_L  S_L( y))(\r^\alpha \trl_L S_L( x)) 
\eeq
and a braided coalgebra map:
\beq\label{S-algmapB}
\Delta_L\circ S_L ( x) = S_L (\r_\alpha \trl_{{L}} \two{ x}) \bot S_L(\r^\alpha \trl_{{L}} \one{ x} ) 
= \r_\alpha \trl_{{L}}  S_L (\two{ x}) \bot \r^\alpha \trl_{{L}}  S_L(\one{ x} )\, .  
\eeq

Due to the quasi-cocommutativity property \eqref{iiR}, the action in \eqref{probot} commutes with the braiding: 
$\trl_{L \bot L} \circ\:\!\Psi_\r=\Psi_\r\circ \trl_{L \bot L}$.
 More generally, given  two $K$-modules $V, W$, the induced action of the coalgebra $K$ on their tensor product commutes with the 
braiding:
$  \Psi_\r:V\otimes W\to W\otimes V$, $\Psi_\r(v\otimes w)=\r_\alpha\trl
w\otimes\r^\alpha\trl v$, that is 
$
k \trl_{V\ot W} \circ \Psi_\r =\Psi_\r \circ k \trl_{V\ot W}
$ for $k \in K$.

\subsection{Braided Lie algebras of derivations}
In order to study braided Lie algebras we take $(K,\r)$ to be
triangular, and not just quasitriangular. This simplifying assumption
is enough for the purposes of the present paper.
A  braided Lie algebra associated with a triangular
  Hopf algebra $(K, \r)$, or simply a $K$-braided Lie algebra, 
is a $K$-module $\g$ with a
bilinear map
$$
[~,~]: \g\otimes \g\to \g
$$
that satisfies the following conditions:
\begin{enumerate}[(i)]
\item 
$K$-equivariance: for $\Delta(k) = \one{k} \ot \two{k}$ the coproduct of $K$,
$$k\trl [u,v]=[\one{k}\trl u,\two{k}\trl v] $$ 
\item braided antisymmetry: $$[u,v]=-[\r_\alpha\trl v, \r^\alpha \trl u ] , $$
\item
braided Jacobi identity: $$[u,[v,w]]=[[u,v],w]+[\r_\alpha\trl v, [\r^\alpha \trl u,w] ] , $$
\end{enumerate}
for all $u,v,w \in \g$, $k\in K$.

As shown in  \cite[\S 5.1]{pgc-main}, the universal enveloping algebra $\U(\g)$ of a braided Lie
  algebra $\g$ associated with $(K,\r)$ is a braided Hopf algebra.   
The coproduct of   $\U(\g)$ is determined requiring the elements of
$\g$ to be primitive,  $\Delta(u)=u\bot 1+1\bot u$,  for all $u\in\g$.
 
Any $K$-module algebra $A$  is a $K$-braided Lie algebra with bracket given by the braided commutator
\beq\label{Abla}
[~,~]: A\ot A\to A, \qquad a\ot b\mapsto [a,b] = ab - (\r_\alpha \trl b) \, (\r^\alpha\trl a) \, .
\eeq 
 (See \cite[Lemma 5.2]{pgc-main}.) 
In particular, if 
$A$ is a $K$-module algebra, then also the $K$-module algebra $(\Hom(A,A), \trl_{\Hom(A,A)})$
of linear maps from $A$ to $A$ with
action
\begin{align}\label{action-hom}
\trl_{\Hom(A,A)}:  K \ot \Hom(A,A) &\to \Hom(A,A) 
\nn \\
k \ot \psi & \mapsto  k \trl_{\Hom(A,A)} \psi : \; A \mapsto \one{k} \trl_A \psi(S(\two{k})\trl A) 
\end{align}
is a  braided Lie algebra with the braided commutator;  here  $S$ is the antipode of $K$.
Elements 
$\psi$ in  $\Hom(A,A)$ which satisfy 
\beq\label{Der}
\psi(aa')= \psi(a)
a' + (\r_\alpha\trl a)\,(\r^\alpha 
\trl_{\Hom(A,A)}
\psi)(a') 
\eeq
for any $a,a'$ in $A$ are called  {braided derivations}. 
We denote
$
{\mathrm{Der}}^\r(A)
$
the  $\bbK$-module
of {braided derivations} of $A$ (to lighten notation we often drop the subscript $\r$).
It is a $K$-submodule of $\Hom(A,A)$, 
with action given by the restriction of 
 $\trl_{\Hom(A,A)}$
 \begin{align}\label{action-der}
\trl_{\Der{A}}:  K \ot \Der{A} &\to \Der{A}
\nn \\
k \ot \psi & \mapsto  k \trl_{\Der{A}} \psi : \; a \mapsto \one{k} \trl \psi(S(\two{k})\trl a) 
\end{align}
and moreover,  see \cite[Prop. 5.7]{pgc-main},  a  braided Lie subalgebra of  $\Hom(A,A)$ with
\begin{align}\label{bracket-der}
[~ , ~]  :& \, \Der{A}\ot \Der{A} \to \Der{A}
\nn \\
& \psi \ot \lambda \mapsto [\psi,\lambda] :=\psi\circ \lambda-(\r_\alpha\trl_{\Der{A}} \lambda) \circ (\r^\alpha\trl_{\Der{A}} \psi).
\end{align}
When the $K$-module algebra $A$ is quasi-commutative, 
that is when
\beq\label{qcAB} 
a\, a'  =   (\r_\alpha \trl{a'}) \, (\r^\alpha\trl{a})~, 
\eeq
 for all $ a,a' \in A$, the  braided Lie algebra $\Der{A}$ 
with
\begin{equation}\label{AmodderA}
  (a \psi)(a'):=a\, \psi(a'),
  \end{equation} 
for  $\psi \in \Hom(A,A)$, $a,a'\in A$, is also a left $A$-submodule of $\Hom(A,A)$.  
The Lie bracket   of $\Der{A}$   satisfies  (\cite[Prop. 5.8]{pgc-main})
\begin{equation}\label{moduloLie}
\begin{split}
 [a \psi, a'\psi'] = a \psi(a') \psi'
&+ a(\r_\alpha\trl a')  [\r^\alpha \trl_{\Der{A}}
    \psi,\psi']\\[.2em]
&    -\r_\beta\r_\alpha\trl a' \big(\r_\delta\r_\gamma \trl_{\Der{A}} \psi'\big)\!\!\:( \r^\delta\r^\beta \trl a) \, 
    \r^\gamma\r^\alpha \trl_{\Der{A}} \psi~
\end{split}
\end{equation}
 for all
$a,a'\in A, \psi,\psi'\in \Der{A}$.

\subsection{Infinitesimal gauge transformations}\label{sec:igt}
Let now $B=A^{co H} \subseteq A $ be a $K$-equivariant Hopf--Galois extension, for $(K,\r)$ a triangular Hopf algebra. 
Inside the braided Lie algebra $\Der{A}$  we consider the subspace of braided derivations that are 
 $H$-comodule maps, 
\begin{align}
\DerH{A} = \{u\in \Der{A} \;|\;  \delta (u(a))=u(\zero{a})\otimes \one{a} \, , \, \,  a \in A \}  \label{LieGG}
\end{align}
and then those derivations that are vertical,
\beq\label{BLieGG}
\aut{A} := \{u\in \DerH{A} \; | \; u(b)=0 \, ,
  \, \,  b\in B\}~. 
\eeq
The linear spaces $\DerH{A}$ and  $\aut{A}$ are $K$-braided Lie subalgebras of $\Der{A}$,  \cite[Prop. 7.2]{pgc-main}. 
Elements of $\aut{A}$ are regarded  as infinitesimal gauge
transformations of the $K$-equivariant Hopf--Galois extension $B=A^{co H} \subseteq A$, \cite[Def. 7.1]{pgc-main}.
There is the corresponding braided Hopf algebra $\U(\aut{A})$ of gauge transformations. 

\subsection{Twisting of braided Lie algebras}\label{sec:TBLA} 
Important   examples of noncommutative principal bundles come from twisting classical structures. 
Aiming at studying their braided Lie algebras of infinitesimal gauge transformations, 
we need to first consider twist deformations of braided Lie algebras.  

We recall some basic results of the theory of Drinfeld twists \cite{drin}.

Let $K$ be a bialgebra (or Hopf algebra). A { twist} for $K$  is an invertible  element 
$\F \in K \ot K$ which is unital, 
$
(\cun \ot \id) (\F)= 1= (\id \ot \cun)(\F) 
$,
and satisfies the twist condition
\beq\label{twist}
(\F \ot 1)[(\Delta \ot \id)(\F)]= (1 \ot \F)[(\id \ot \Delta) (\F)] ~.
\eeq
 For $\F$ and its inverse $\bF$ we write $\F=\F^{ \alpha} \ot \F_\alpha$ and $\bF=: \bF^{\alpha} \ot
\bF_{\alpha}$, with an implicit sum. 
\noindent
The $R$-matrix  $\r$ of a quasitriangular bialgebra $K$ is a twist for $K$.

When $K$ has a twist it can be endowed with a second bialgebra structure which is obtained by deforming its coproduct and leaving its counit and multiplication unchanged. Moreover if $K$ is triangular, or more in general quasitriangular, so is the new bialgebra: 
\begin{prop}\label{prop:twist}
Let $\F=\F^{ \alpha} \ot \F_\alpha$  be a twist on a bialgebra  $(K,  m, \eta, \Delta, \cun)$. Then
the algebra $(K,  m, \eta)$ with coproduct 
\beq\label{cop-twist}
\Delta_\F(k):= \F \Delta(k) \bF= \F^{\alpha} \one{k} \bF^{\beta} \ot \F_\alpha \two{k} \bF_{\beta}~, \qquad k \in K
\eeq
and counit $\cun$  is a bialgebra. If in addition $K$ is a Hopf algebra, then the
twisted bialgebra $\KF:=(K,  m, \eta, \Delta_\F, \cun)$
is a Hopf algebra with antipode $S_\F(k):=\uf S(k) \buf ,
$ where $\uf$ is the invertible element
$\uf:=\F^{\alpha} S(\F_\alpha)$ with $\buf= S(\bF^{\alpha})\bF_{\alpha}$ its inverse.

\noindent
Finally, if 
 $(K, \r)$ is a quasitriangular bialgebra  (a Hopf algebra), such is the twisted bialgebra  (Hopf algebra) $\KF$ with  $R$-matrix
\beq\label{urm-def}
\rF := \F_{21} \, \r \, \bF = \F_\alpha \r^\beta \bF^{\gamma} \ot \F^\alpha \r_\beta \bF_{\gamma} 
\eeq
and inverse   $\brF :=\F \, \br \, \bF_{21} = \F^\alpha \br^\beta \bF_{\gamma} \ot \F_\alpha \br_\beta \bF^{\gamma}$. If  $(K, \r)$ is  triangular, so is  $(\KF,\rF)$.
 \end{prop}

Any  $K$-module   $V$ with left action $\trl_V : K \ot V \to V$, is also a $\KF$-module
with the same linear map  $\trl_V$, now thought as a map 
$\trl_V: \KF \ot V \to V $. 
When thinking of $V$ as a $\KF$-module we denote it by $V_\F$, with action $\trl_{V_F}$. 
Moreover, any $K$-module morphism $\psi : V\to W$ 
can be thought of as a morphism $\psi_F  : V_\F \to W_\F$.  

If $A$ is a $K$-module algebra, with multiplication $m_A$ and unit $\eta_A$, in order for the action $\trl_{A_\F}$ to be an algebra map one has to endow  the $\KF$-module
$A_\F $   with a  new 
algebra
structure:  the unit is unchanged,  while the product is deformed to
\begin{align}\label{rmod-twist} 
 m_{A_\F} : A_\F \otF A_\F \, \longrightarrow A_\F~ , \qquad 
a\otF a^\prime \, \longmapsto a \dotF a':= (\bF^{\alpha} \trl_A a)\, (\bF_{\alpha} \trl_A a') .
\end{align}
For any $K$-module algebra map $\psi : A\to A^\prime$, the $\KF$-module map    
$\psi : A_\F \to A^\prime_\F$ is an algebra map  for
 the deformed products. 

If $C$ is a $K$-module coalgebra, the $\KF$-module $C_\F$ is a $K_F$-module coalgebra with 
counit $\varepsilon_\F=\varepsilon$ as linear map and coproduct 
 \beq\label{comod-twist}
\Delta_\F: C_\F\to C_\F\otimes_\F C_\F~,
  ~~c\mapsto \Delta_\F(c)=\F^\alpha\trl \one{c}\otimes_\F\F_\alpha\trl
  \two{c}~.
  \eeq
  The twist
  $L_\F$ of
  a braided Hopf algebra $L$ is obtained twisting $L$ as a $K$-module
  algebra and as a $K$-module coalgebra, cf. \cite[ Prop.~4.11]{pgc-main}.

We next recall that the action of a braided  Hopf algebra (or just bialgebra) 
$L$ on a $K$-module algebra $A$ is a $K$-equivariant action $\btrl_A:L\ot A\to A$
 which satisfies 
\beq\label{LAbraidedaction}
x \btrl_A (a a')=(\one{x}\btrl_A (\r_\alpha \trl_A  a))\, ((\r^\alpha \trl_L \two{x})\btrl_A a') ~,
\eeq
for all $a,a'\in A$.
When twisting this leads to an action
\beq\btrl_{A_\F}: L_\F\otimes_\F A_\F\to
A_\F~,~~x\btrl_{A_\F} a=(\bF^\alpha\trl_L x)\btrl_{A} (\bF_\alpha \trl_A a)~.
\eeq

When $\g$ is a braided Lie algebra associated with a triangular Hopf algebra
$(K,\r)$, and $\F$ is a twist for $K$, the $K_\F$-module $\g_F$
  inherits from $\g$ a twisted bracket (\cite[Prop. 5.14]{pgc-main}):
\begin{prop}\label{prop:gf}
The $K_\F$-module $\g_\F$ with bilinear map 
\beq 
[~ , ~]_\F  =  \g_\F\otimes \g_\F\to \g_\F~,~~u\otimes v\to [u,v]_\F :=[\bF^\alpha \trl_\g u,\bF_\alpha\trl _\g v]~
\eeq
is a 
braided Lie algebra associated with $(K_\F, \r_\F)$.
\end{prop}

As a particular case of the above, consider the braided Lie algebra $\big(\Der{A}, [\, , \, ] \big)$, for $A$ a $K$-module  algebra. It  consists of the $K$-module of braided derivations of  $A$ to itself, with action 
$\trl_{\Hom(A,A)}$ as in \eqref{action-der}, and  bracket the braided commutator \eqref{bracket-der}.  It is a braided Lie algebra   associated with the triangular Hopf algebra $(K, \r)$.

On the one hand, we obtain the braided Lie algebra $(\Der{A}_\F) , [\, ,\, ]_\F)$ associated with the triangular  Hopf algebra $(K_\F, \r_\F)$. The $K_\F$-action $\trl_{\Der{A}_F}$ coincides with  $\trl_{\Der{A}}$  as linear map.
 The Lie bracket is given by the braided commutator
$$
[\psi,\lambda]_\F = \psi \circ_\F \lambda -  ( \rF_\alpha  \trl_{\Der{A}} \lambda)  \circ_\F (  \rF^\alpha \trl_{\Der{A}} \psi)\; ,
$$
with the composition (in fact in $(\Hom(A,A), \circ)$) that is changed as in \eqref{rmod-twist}:
\beq\label{circF}
\psi \circ_\F \phi = (\bF^{\alpha} \trl_{\Der{A}} \psi)\circ (\bF_{\alpha} \trl_{\Der{A}} \phi) \, .
\eeq  

On the other hand, there is  the braided Lie algebra $\Der{A_\F}$ of the $K_\F$-module  $A_\F$ 
associated with   $(K_\F, \r_\F)$.
We use the notation $\Delta_\F (k)=: \fone{k} \ot \ftwo{k}$ for the coproduct in $\KF$ to distinguish it from the original one $\Delta(k)= \one{k} \ot \two{k}$ in $K$.
 The $K_\F$-action is 
\begin{align}\label{action-homF}
\trl_{\Der{A_\F}}:  K_\F \ot \Der{A_\F} &\to \Der{A_\F} \\
k \ot \psi & \mapsto  k \trl_{\Der{A_\F}} \psi : \; a \mapsto \fone{h} \trl_{A_\F} \psi(S_\F(\ftwo{h})\trl_{A_\F} a) .
\nn \end{align}
with bracket 
$$
[\psi,\lambda]_{\rF} = \psi \circ \lambda -  ( \rF_\alpha  \trl_{\Der{A_\F}} \lambda)  
\circ ( \rF^\alpha \trl_{\Der{A_\F}} \psi)\; .
$$
These two braided Lie algebras are isomorphic
\cite[Thm. 5.19]{pgc-main}:
\begin{thm}\label{thm:Dalg}
  The braided  Lie algebras $(\Der{A}_\F , [\, ,\, ]_\F )$ and
$(\Der{A_\F} , [\, ,\, ]_{\rF})$ are isomorphic via the map
\begin{align}\label{mappaD}
\dd: \Der{A}_\F \to  \Der{A_\F} , \qquad
\psi \mapsto \dd(\psi): a \mapsto (\bF^\alpha \trl_{\Der{A}_\F} \psi)(\bF_\alpha \trl_A a) \; ,
\end{align}
which satisfies 
$\dd \big([\psi,\lambda]_\F \big) = [\dd (\psi) , \dd( \lambda) ]_{\rF}$, for all $\psi,\lambda \in \Der{A}_\F$. 
It has inverse
\begin{align}\label{mappaDinv}
\dd^{-1}: \Der{A_\F} \to \Der{A}_\F , \qquad
\psi \mapsto 
\dd^{-1}(\psi): a \mapsto (\F^\alpha  \trl_{\Der{A_\F}} \psi)(\F_\alpha \trl_{A_\F}  a) \; .
\end{align}
\end{thm}

This isomorphism extends as  algebra map to the
universal enveloping algebras
$$
\dd:\U(\Der{A}_\F)\to\U(\Der{A_\F})
$$
resulting into a braided Hopf algebra isomorphism.
We further have the braided Hopf algebras isomorphisms  (see  \cite[Prop.~5.18]{pgc-main})
$\U(\Der{A})_\F \simeq\U(\Der{A}_\F)\simeq\U(\Der{A_\F})$.

\begin{rem}
As shown in \cite{pgc-main}, the isomorphism $\dd:\Der{A}_\F \to
\Der{A_\F}$ is the restriction of a more general isomorphism 
$\dd:  ((\Hom(A,A)_\F, \circ_\F), [\, ,\, ]_\F)\to ((\Hom(A_\F,A_\F), \circ) ,[\, , \,]_{\rF})$. This 
result indeed holds in more generality for $A$ just a $K$-module  and not necessarily  a $K$-module algebra.
\end{rem}

As mentioned, 
when $A$ is quasi-commutative the $K$-braided Lie
algebra $\Der{A}$ has an $A$-module structure defined in \eqref{AmodderA}
that is compatible with the Lie bracket of $\Der{A}$.

The $K_\F$-braided Lie
algebra $\Der{A}_\F$ has $A_\F$-module
structure  
\begin{equation}\label{acdotFD}
  a\cdot_\F\psi:=(\bF^{\alpha}
  \trl_A a)\, (\bF_{\alpha}\trl_{\Der{A}} \psi) \, ,
\end{equation}
for all $\psi \in \Der{A}_\F$ and $a\in A_\F$.
The compatibility of the braided bracket with this module structure then, for all $\psi,\psi'\in \Der{A}_\F$, $a,a'\in A_\F$, reads 
\begin{align}\label{LieAmodF+}
  [a\cdot_\F\psi , &
  a'\cdot_\F\psi']_\F= ~a\cdot_\F[\psi,a']_\F\cdot_\F\psi'+
  a\cdot_\F({\r_\F}_\alpha\trl_{A_\F}
  a')\cdot_\F[{\r_\F}^\alpha \trl_{\Der{A}_\F} \psi,\psi']_\F \\ 
  &
  - {\r_\F}_\beta {\r_\F}_\alpha\trl_{A_\F} a' \cdot_\F[{\r_\F}_\delta {\r_\F}_\gamma\trl_{\Der{A}_\F} \psi'\,,\, {\r_\F}^\delta {\r_\F}^\beta \trl_{A_\F} a]_\F\cdot_\F
    {\r_\F}^\gamma  {\r_\F}^\alpha\trl_{\Der{A}_\F} \psi~. \nn
\end{align}
Here an element in $A$ is thought as a linear map $A\to A$ given by left multiplication. Then 
 $[\psi,a]_\F =
  [\bF^\alpha\trl_{\Der{A}} \psi,\bF_\alpha\trl_A a]=
 (\bF^\alpha\trl_{\Der{A}}
 \psi)(\bF_\alpha\trl_A a)$.

Also the $K_\F$-braided Lie
algebra  $\Der{A_\F}$ has compatible $A_\F$-module structure. With the product $\dotF$ in \eqref{rmod-twist} 
this is given as in \eqref{AmodderA} by
\beq\label{LieAFmod}
(a \dotF \psi)(a')=a \dotF \psi(a') 
\eeq
for any $a,a'\in A_\F,\psi \in \Der{A_\F}$.

The isomorphism $\dd:\Der{A}_\F\to \Der{A_\F}$  respects the  $A_\F$-module
structures:

\begin{cor}\label{DAmodiso}
If the $K$-module algebra $A$ is quasi-commutative, the braided Lie
algebra isomorphism $\dd: (\Der{A}_\F , [\, ,\, ]_\F )\rightarrow (\Der{A_\F} , [\, ,\, ]_{\rF})$ of Theorem \ref{thm:Dalg} is also an isomorphism of
the $A_\F$-modules $\Der{A}_\F$ and $\Der{A_\F}$:
$$
\dd (a \cdot_\F\psi) = a \dotF \dd(\psi) ,
$$
for $a\in A_\F$ and $\psi \in \Der{A}_\F$.
\end{cor}

Next, let $B \subseteq A$ be a $K$-equivariant Hopf--Galois
  extension. We use the above isomorphisms for
the  $K$-braided Lie algebra of derivations $\Der{A}$ and its braided subalgebras 
$\DerH{A}$ and  $\aut{A}$  defined in
\eqref{LieGG} and in \eqref{BLieGG}.

The $K$-braided Lie algebras $(\aut{A}, [~,~])\subseteq (\DerH{A}, [~,~])$
are twisted to the
$K_\F$-braided Lie algebras $(\aut{A}^{}_\F,[~,~]_\F)\subseteq (\DerH{A}_\F, [~,~]_\F)$ with bracket
$[~,~]_\F$. These are braided Lie subalgebras of $(\Der{A}_\F, [~,~]_\F)$. We can equivalently consider the $K_\F$-braided Lie
algebras $(\autF{A_\F},[~,~]_{\r_\F})\subseteq (\DerHF{A_\F}, [~,~]_{\r_\F})$ that are braided Lie subalgebras of $(\Der{A_\F}, [~,~]_{\r_\F})$.
These are isomorphic, \cite[Prop. 8.1]{pgc-main}: 
\begin{prop}\label{autautF}
The isomorphism $\dd: (\Der{A}_\F, [~,~]_\F){\, \to \,} (\Der{A_\F}, [~,~]_{\r_\F})$ of braided
Lie algebras in Theorem \ref{thm:Dalg} restricts to 
isomorphisms
$\dd: \DerH{A}_\F{\, \to \,}\DerHF{A_\F}$ and
$\dd: \aut{A}^{}_\F{\, \to \,}\autF{A_\F}$ 
of $(K_\F,\r_\F)$-braided Lie algebras.
\end{prop}

In \S \ref{sec:examples} we work out the braided Lie algebra of  equivariant
derivations and of infinitesimal gauge transformations for 
two important examples of principal bundles over the noncommutative
$4$-sphere $S^4_\theta$ of \cite{cl01}. 
 We use the general theory developed in this section to
obtain the braided Lie algebra of  equivariant derivations and of
infinitesimal gauge transformations of these noncommutative bundles
from their classical counterparts.

These noncommutative principal bundles are $*$-algebra extensions
(which can be completed to $C^*$-algebras).
The $*$-structures canonically lift to the Lie algebras of braided
derivations and of gauge transformations. 
Before considering these examples, in the next
section we proceed with a systematic analysis of $*$-structures for
braided Hopf and Lie algebras.
The results of \S \ref{sec:examples} are however presented in a
  self contained way so that  \S  \ref{sec:*} might be skipped in a
first reading.

\section{Braided Hopf and Lie $*$-algebras}\label{sec:*}

In this section the ground field is $\bbK=\mathbb{C}$.
We present a study of compatibility conditions for
defining $*$-structures on Hopf algebras and their representations.
The study of braided Hopf $*$-algebra actions on braided $*$-algebras
associated with quasitriangular Hopf algebras is new to the best of
our knowledge.

A $*$-structure  on a vector space $V$ is an antilinear involution
$*:V\to V$, $v \mapsto v^*$, on an
algebra $A$ one also requires $*:A\to A$ to be antimultiplicative. A $*$-structure on a Hopf
algebra $K$ is a $*$-structure on the
algebra $K$ that satisfies $\Delta(k^*)=\Delta(k)^{*\otimes *}$ for
all $k\in K$; it then follows that $\varepsilon(k^*)=\overline{\varepsilon(k)}$ and $(S\circ *)^2=\id$. In particular $S$ is invertible, with $S^{-1}=  * \circ  S\circ *$.
If $V$ is a $K$-module with a $*$-structure, one requires the compatibility condition
\begin{equation}\label{trlstar}
(k\trl_V v)^*=S^{-1}(k^*)\trl_V v^*
\end{equation}
 for all $k\in K, v \in V$.    
This  condition  is well defined: it is equivalent to require that $\succ_V:
  K\otimes V\to V$,  defined by
\begin{equation}\label{succ}
   k\suc_V v:= (S^{-1}(k^*)\trl_V v^*)^* , 
\end{equation}
is an action of the Hopf algebra $K$ on  $V$ that coincides with the starting one $ \trl_V$. Indeed
\begin{align}  \label{Kact} 
k\suc (h\suc v)&=k\suc(S^{-1}(h^*)\trl v^*)^*=(S^{-1}(k^*)S^{-1}(h^*)\trl v^*)^*
  =kh\suc v  .
\end{align}
The condition \eqref{trlstar}  is also required for $A$ a $K$-module $*$-algebra. In this case \eqref{trlstar} is also well defined with respect to 
the multiplication in $A$. 
Indeed, $k \trl_L 1_A = \varepsilon (k)1_A$ and \begin{align}
  (\one{k}\suc a)(\two{k}\suc b)&=\big((S^{-1}(\two{k}^*)\trl
                                    b^*)(S^{-1}(\one{k}^*)\trl
                                    a^*)\big)^*\nn\\
  &=
    \big(S^{-1}({k}^*)\trl (b^*a^*)\big)^* =k\suc (ab)~\label{KonAlg}
  \end{align}
for all $k,h\in K, \,a, b\in A$. Here and in the following to lighten the
  notations we frequently omit the subscript on the actions. 
 \begin{ex}
 Any Hopf $*$-algebra $K$ with adjoint action  $K\otimes K\to
 K$, $k\!\!\:\trl\!\!\: k'=\one{k}k'S(\two{k})$ is a $K$-module $*$-algebra. 
This motivates the definition \eqref{trlstar} following the
conventions  in \cite{KS} rather than those in \cite{Majid}.
\end{ex}
 \begin{ex}
 Condition \eqref{trlstar} is dual to that for a comodule $*$-algebra. 
Given a $*$-algebra $A$ which is a right comodule algebra for a Hopf $*$-algebra $\mathcal{U}$, with coaction
$\delta: A \to A \ot \mathcal{U}$, $a \mapsto \zero{a} \ot \one{a}$, one requires $\delta(a^*)= (\zero{a})^* \ot (\one{a})^*$.   If $K$ and $\mathcal{U}$ are dually paired Hopf $*$-algebras, one has $\langle k, u^* \rangle = \overline{\langle S^{-1}(k^*), u \rangle}$.  Then $A$ is a module $*$-algebra with $K$-action $k \trl a = \zero{a} \langle k, \one{a} \rangle$ satisfying \eqref{trlstar}.
 \end{ex}
In the present paper we deal with braided Hopf and Lie
algebras associated with a (quasi)triangular Hopf algebra $(K, \r)$. 

When $K$ is quasitriangular we require its $R$-matrix  to be
  antireal, that is, $\r^{*\otimes *}=\br$. When $K$ is triangular
  this condition coincides  with the reality condition  $\r^{*\otimes *}=\r_{21}$. 

\subsection{Braided Hopf $*$-algebras}\label{sec:*h}
In this section we take $(K, \r)$ quasitriangular with $\r$ antireal. We use the braiding $\Psi_\r:L\bot L\to L\bot L, \Psi_\r(x\bot
  y)=\r_\alpha \trl_L y\bot \r^\alpha\trl_L x$, in \eqref{braiding} to induce the 
$*$-structure from a $K$-module $*$-algebra $L$ to the $K$-module algebra $L\bot L$ defined 
in \S\ref{se:bha} (and to more general tensor products of $K$-module $*$-algebras).

\begin{lem} \label{lem:starbot}
Given a $K$-module $*$-algebra $L$, the $K$-module algebra $L \bot L$
with product $\tpr$ in \eqref{tpr} and involution $(x \bot y)^*
  = \Psi_\r(y^*\bot x^*)$, that is, 
\beq\label{starbot}
(x \bot y)^* =(\r_\alpha \trl_L x^*) \bot (\r^\alpha \trl_L y^*)~,
\eeq 
 for $x,y\in L$,
is a $K$-module $*$-algebra.
\end{lem}
\begin{proof}
 The matrix $\r$ being antireal   implies
that \eqref{starbot} is an involution. 
It is antimultiplicative: using  \eqref{trlstar} for $\trl_L$
together with antireality and properties \eqref{ant-R} of the $\r$-matrix
we compute (omitting the subscript on the actions)
\begin{align*}
\left((x \bot   y) \tpr (  x' \bot   y')\right)^* 
&=
\left(  x(\r_\alpha \trl   x') \bot (\r^\alpha \trl   y)   y' \right)^*
\\
& =
\r_\beta \trl \left( x(\r_\alpha \trl   x' ) \right)^*  \bot \r^\beta \trl  \left((\r^\alpha \trl y)   y' \right)^*
\\
& =
\r_\beta \trl \left( (\r_\alpha \trl   x' )^* x^* \right)   \bot \r^\beta \trl  \big({ y'}^* (\r^\alpha \trl y)^*    \big)
\\
& =
\r_\beta \trl \left( (\br_\alpha \trl   {x'}^*) x^* \right)   \bot \r^\beta \trl  \big({ y'}^*\:\! \br^\alpha \!\!\:\trl y^*    \big)
\\
& =
 (\one{\r_\beta}  \br_\alpha \trl   {x'}^* ) (\two{\r_\beta}\trl  x^*)    
 \bot 
  ( \one{\r^\beta} \trl { y'}^* )( \two{\r^\beta}  \br^\alpha \trl y^*   )
\\
& =
 ({\r_\beta} \r_\tau \br_\alpha \trl   {x'}^* ) ( {\r_\gamma} \r_\mu \trl  x^*)    
 \bot 
  ( \r^\gamma {\r^\beta} \trl{ y'}^* )(  \r^\mu \r^\tau  \br^\alpha \trl y^*   )
\\
&=( \r_\beta \trl {x'}^*) (\r_\gamma \r_\mu \trl x^*) \bot (\r^\gamma \r^\beta \trl {y'}^* )( \r^\mu \trl y^* \big) 
\\
&=\big( \r_\beta \trl {x'}^* \bot \r^\beta \trl {y'}^* \big) \tpr  \big( \r_\mu \trl x^* \bot \r^\mu \trl y^* \big) 
\\
& = (  x' \bot   y')^*\tpr (x \bot   y) ^*~.
\end{align*}
Finally, we show the compatibility condition \eqref{trlstar} between the $*$-structures. Recalling that $L \bot L$  has action \eqref{probot}, we compute
\begin{align}
(k \trl (x \bot y))^* 
&=  (\one{k} \trl x \bot \two{k} \trl y)^*
\nn\\
&=  \r_\alpha \trl (\one{k} \trl x)^* \bot \r^\alpha \trl (\two{k} \trl y)^*
\nn\\
&=  \r_\alpha S^{-1}(\one{k}^*) \trl x^* \bot \r^\alpha S^{-1}(\two{k}^*) \trl y^*
\nn\\
&=  \r_\alpha \two{(S^{-1}(k^*))} \trl x^* \bot \r^\alpha \one{(S^{-1}(k^*))} \trl y^*
\nn\\
&=   \one{(S^{-1}(k^*))} \r_\alpha \trl x^* \bot   \two{(S^{-1}(k^*))} \r^\alpha\trl y^*
\nn\\
&=   S^{-1}(k^*) \trl \big( \r_\alpha \trl x^* \bot    \r^\alpha\trl y^* \big)
\nn\\
&=   S^{-1}(k^*) \trl \big(  x \bot     y \big)^* \label{Rsucc}
\end{align}
where for the third last equality we used the quasi-cocommutative condition \eqref{iiR}.
\end{proof}
  Considering the last and third expression in \eqref{Rsucc} one has
  $$
 \left[ S^{-1}(k^*) \trl \big(  x \bot     y \big)^* \right]^*  =  \left[ \r_\alpha S^{-1}(\one{k}^*) \trl x^* \bot \r^\alpha S^{-1}(\two{k}^*) \trl y^* \right]^*
  $$
and, recalling the action \eqref{succ}, this reads 
$k\suc_{L\bot L} (x\bot y)=\one{k}\suc_L x\,\bot\, \two{k}\suc_L y$. 
 This
  proves that $\suc_{L \bot L}$ is an action of the coalgebra $K$ on $L\bot L$.
\begin{lem} The action $\suc : K\otimes L\to L$ commutes with
the braiding isomorphism $\Psi_\r:L\bot L\to L\bot L$, $\Psi_\r(x\bot y)=\r_\alpha\trl
  y\bot\r^\alpha\trl x$ of the original  $K$-action $\trl$,
 \begin{equation}\label{Psisuc}
    k\!\suc \!(\Psi_\r(x\bot y))=\Psi_\r(
    k\!\suc\! (x\bot y))~,
    \end{equation}for
  all $k\in K, x, y\in L$. 
\end{lem}
\begin{proof}
We first compute, using the  antireality of the $R$-matrix,
  \begin{align} \label{bracketR*}
(\r^\alpha \trl x )^* \bot (\r_\alpha\trl y)^* 
&=  S^{-1} ({\r^{\alpha\:\!}}^*) \trl x^*\bot S^{-1}({\r_\alpha}^{\!*}) \trl y^*
\\
&
=  S^{-1} (\br^\alpha) \trl x^*\bot S^{-1}(\br_\alpha)  \trl y^*
=   \br^\alpha \trl x^*\bot  \br_\alpha  \trl y^* ,
 \end{align}
for all $x,y\in \g$. This and the $*$-structure of $L \bot L$ in \eqref{starbot} prove
the lemma:
  \begin{align*}
  k\!\suc \!(\Psi_\r(x\bot y))&
                                   = k\suc(\r_\alpha \trl y\!\:\bot\!\:
                                \r^\alpha\trl x)\\
                              &=[S^{-1}(k^*) \trl(\r_\alpha\trl
                                y\!\:\bot\!\:\r^\alpha\trl x )^*]^*\\
                              &=[S^{-1}(k^*)\trl(\r_\beta\trl(\r_\alpha\trl
                                y)^*\!\:\bot\!\:\r^\beta\trl(\r^\alpha\trl
                                x)^*)]^*\\
                                 &=[S^{-1}(\two{k}^*)\trl y^*\bot
                                   S^{-1}(\one{k}^*)\trl x^*]^*\\
  &=\r_\alpha\trl (S^{-1}(\two{k}^*)\trl y^*)^*\bot
    \r^\alpha\trl(S^{-1}(\one{k}^*)\trl x^*)^*\\
  &=\r_\alpha\trl (\two{k}\suc y)\bot
    \r^\alpha\trl(\one{k}\suc x)\\
  &=\Psi_\r(\one{k}\suc x\!\:\bot\!\: \two{k}\suc y)~.\\[-3em]
  \end{align*}
\end{proof}

\begin{defi}\label{*HA} A $*$-structure on a  braided Hopf  algebra 
 $L$  associated with $(K,\r)$ is a $*$-structure on the $K$-module
 algebra $L$ such that the braided coproduct $\Delta_L: L\to L \bot L$
 is a $*$-algebra map,
$\Delta_L(x^*) = (\Delta_L (x))^*$ with $*$-structure on $L \bot L$ in \eqref{starbot}. 
\end{defi}

The $*$-algebra map condition for $\Delta_L$  is equivalent to $\Delta_L = *\circ\Delta_L\circ *$. This is well defined
since $*\circ\Delta_L\circ *: L\to  L \bot L$ is a coassociative $K$-module map and
an algebra map. The $K$-equivariance follows from that of
$\Delta_L$ and the compatibility \eqref{trlstar}. The $*$-algebra map
property is straighforward and coassociativity is verified by direct computation.

The (braided) antipode $S_L$  of a braided Hopf $*$-algebra satisfies $S_L \circ * \circ S_L \circ * =\id$ and so is invertible. 
Using \eqref{S-algmapB},  $\Delta_L \circ S_L = (S_L \ot S_L ) \circ \Psi_\r \circ \Delta_L$, one gets:
$$
\Delta_L \circ S_L^{-1} = (S_L^{-1} \ot S_L^{-1} ) \circ \overline{\Psi}_\r \circ \Delta_L   
$$
with $\overline{\Psi}_\r$ the inverse of $\Psi_\r$: 
$ \overline{\Psi}_\r(x\ot y)=\br^\alpha \trl_L y\ot \br_\alpha\trl_L x$, for $x,y \in L$. Then, using \eqref{starbot},   
one gets
$$
\Delta_L (S_L(x^*)) = \r_\alpha \r^\beta \trl S_L((\two{x})^*) \bot \r^\alpha \r_\beta \trl S_L((\one{x})^*)  \; ,
$$
for each $x \in L$, together with 
\beq\label{eqS*}
\Delta_L (S^{-1}_L(x^*)) =  S^{-1}_L((\two{x})^*) \bot  S^{-1}_L((\one{x})^*)  \;  .
\eeq
 A braided Hopf $*$-algebra $L$ acts on a $K$-module
$*$-algebra $A$, with action $\btrl_A:L\otimes A\to A$ satisfying \eqref{LAbraidedaction},
if the
$*$-structure of $A$ satisfies the compatibility condition  
\beq \label{comp**}
( x \btrl_A a)^* =  (  \br^\alpha \trl_L S^{-1} (x^*) ) \btrl_A (\br_\alpha \trl_A a ^* )~ 
\eeq 
that generalizes condition \eqref{trlstar}.
\begin{prop} \label{prop:*actionbH}
The compatibility condition \eqref{comp**} is well-defined. 
\end{prop}
\begin{proof}
 The  compatibility condition is equivalent to require that 
$$
x \suc_{{A}} a :=  \left[ (  \br^\alpha \trl_L S^{-1} (x^*) ) \btrl_A (\br_\alpha \trl_A a ^* ) \right]^* 
$$
is an action of the braided Hopf algebra $L$ on the $K$-module algebra $A$, which coincides with the starting action $\btrl_A$.
We need to show that the map $\suc$ defines a $K$-equivariant action that satisfies \eqref{LAbraidedaction}. 
 To lighten the notation we omit the subscript on the actions.  Firstly, $\suc$ is an action.
The condition $x \suc 1_A = \cun_L(x) 1_A$ follows from the unitality
of $\r$.   
Observing that the compatibility of the inverse
  antipode with the multiplication uses the inverse braiding,
$
S^{-1}_L (xy)  = ( \br^\alpha \trl_L  S^{-1}_L( y))(\br_\alpha \trl_L S^{-1}_L( x)) 
$,  cf. \eqref{S-algmapA},
we compute
\begin{align*}
 (xy) \suc  v  &=  \left[(  \br^\alpha \trl S^{-1} ( (xy)^* ) ) \btrl (\br_\alpha \trl v^*) \right]^*
\\
&=  \left[\left(  \br^\alpha \trl ((\br^\gamma \trl S^{-1} (x^*))(\br_\gamma \trl S^{-1} (y^*))  \right) \btrl (\br_\alpha \trl v^*) \right]^*
\\
&= \left[ \left(  (\one{\br^\alpha} \br^\gamma \trl S^{-1} (x^*))(\two{\br^\alpha} \br_\gamma \trl S^{-1} (y^*))  \right) \btrl (\br_\alpha \trl v^*) \right]^*
\\
&= \left[ \left(  ( \br^\beta \br^\gamma \trl S^{-1} (x^*))( \br^\alpha \br_\gamma \trl S^{-1} (y^*))  \right) \btrl (\br_\alpha \br_\beta \trl v^*) \right]^*.
\end{align*}
For the last equality we used the analogous of property \eqref{iii} for $\br$.
On the other hand
\begin{align*}
x \suc (y \suc v) &= x\suc   \left[  (  \br^\alpha \trl S^{-1} (y^*) ) \btrl (\br_\alpha \trl v^*) \right]^*   
\\
&=  \left[ (\br^\beta  \trl S^{-1}(x^*)) \btrl \left( \br_\beta \trl \left[ (  \br^\alpha \trl S^{-1} (y^*) ) \btrl (\br_\alpha \trl v^*)\right]  \right) \right]^* 
\\
&=  \left[ (\br^\beta  \trl S^{-1}(x^*))   ( \one{\br_\beta} \br^\alpha \trl S^{-1} (y^*) ) \btrl (\two{\br_\beta}\br_\alpha \trl v^*)    \right]^*   
\\
&=  \left[ (\br^\beta  \br^\gamma \trl S^{-1}(x^*)) (  {\br_\beta} \br^\alpha \trl S^{-1} (y^*) ) \btrl ( {\br_\gamma}\br_\alpha \trl v^*)    \right]^*   ,
\end{align*}
having used  property \eqref{iii}  again. The two expressions then coincide due to the quantum Yang--Baxter equation
$
\br_{12} \br_{13} \br_{23}= \br_{23} \br_{13}\br_{12} 
$. 
Next we show $K$-equivariance:
\begin{align*}
k \trl( x \suc v) &=  k \trl \left[ (  \br^\alpha \trl S^{-1} (x^*) ) \btrl (\br_\alpha \trl v ^* ) \right]^* 
\\
&= \left[S^{-1}(k^*) \trl \left[ (  \br^\alpha \trl S^{-1} (x^*) ) \btrl (\br_\alpha \trl v ^* ) \right] \right]^*
\\
&=  \left[ \big( \one{( S^{-1}(k^*)}    \br^\alpha) \trl S^{-1} (x^*) \big) \btrl \big( \two{( S^{-1}(k^*)}  \br_\alpha ) \trl v^* \big)  \right]^*
\\
&= \left[  \big(    \br^\alpha \two{ S^{-1}(k^*)}  \trl S^{-1} (x^*) \big) \btrl \big(  \br_\alpha \one{ S^{-1}(k^*)}  \trl v^* \big)  \right]^*
\\
&=  \left[  \big(    \br^\alpha   S^{-1}(\one{k}^*)  \trl S^{-1} (x^*) \big) \btrl \big(  \br_\alpha S^{-1}(\two{k}^*)  \trl v^* \big)  \right]^*
\\
&=\left[ (  \br^\alpha \trl S^{-1} (S^{-1}(\one{k}^*) \trl x^* )) \btrl (\br_\alpha   S^{-1}(\two{k}^* )\trl v^*)  \right]^* 
\\
&=\left[ (  \br^\alpha \trl S^{-1} ((\one{k} \trl x) ^* )) \btrl (\br_\alpha \trl  (\two{k} \trl v)^* ) \right]^* 
\\
& = (\one{k} \trl x) \suc (\two{k} \trl v) 
\end{align*}
where for the fourth equality we used the co-commutativity \eqref{iiR} of the coproduct of $K$ and for the sixth equality the $K$-equivariance of  the braided antipode $S$ of $L$.

 We are left  to show that  \eqref{LAbraidedaction} holds: using the same property for the action $\btrl$
  and property \eqref{eqS*} for the inverse of the braided antipode, we compute
\begin{align*}
& x \succ (a a')
\\
& = \left[ (  \br^\alpha \trl S^{-1} (x^*) ) \btrl \big( (\one{\br_\alpha} \trl {a'}^*)( \two{\br_\alpha} \trl a^* ) \big)\right]^* 
\\
& = \left[ \big[\one{  (  \br^\alpha \trl S^{-1} (x^*) \big)} \btrl   (\r_\gamma \one{\br_\alpha} \trl {a'}^*)\big] \big[   \big( \r^\gamma \trl \two{  \big(  \br^\alpha \trl S^{-1} (x^*) \big)}  \big) \btrl( \two{\br_\alpha} \trl a^* ) \big] \right]^* 
\\
& = 
\big[
 \big( \r^\gamma   \trl \two{\big(\br^\alpha \br^\mu \trl S^{-1} (x^*) \big)}  \big) \btrl  (\br_\mu \trl a^* ) \big]^* 
 \big[ \one{\big(  \br^\alpha \br^\mu \trl S^{-1} (x^*) \big)} \btrl  ( \r_\gamma \br_\alpha \trl {a'}^*) \big]^* 
\\
 & =  \big[  \big( \r^\gamma \two{\br^\alpha} \two{\br^\mu} \trl S^{-1} ({\one{x}}^*) \big)   \btrl ( \br_\mu \trl a^* ) \big) \big]^*  \big[ (\one{  \br^\alpha} \one{ \br^\mu} \trl S^{-1} ({\two{x}}^*) )  \btrl  (\r_\gamma \br_\alpha \trl {a'}^*) \big]^* 
\\
 & = \big[(\r^\gamma  \br^\alpha  {\br^\mu} \trl S^{-1} ({\one{x}}^*) )   \btrl ( \br_\mu \br_\tau \trl a^* ) \big)\big]^*  
\big[ (\br^\nu   \br^\tau \trl S^{-1} ({\two{x}}^*) )  \btrl  (\r_\gamma \br_\alpha \br_\nu \trl {a'}^*) \big]^*
 \\
 & = 
 \big[  ( {\br^\mu} \trl S^{-1} ({\one{x}}^*) )   \btrl ( \br_\mu \br_\tau \trl a^* ) \big)\big]^* 
  \big[ (\br^\nu   \br^\tau \trl S^{-1} ({\two{x}}^*) )  \btrl  (  \br_\nu \trl {a'}^*) \big]^*
 \end{align*}
 On the other hand, also using the antireality of $\r$,
\begin{align*}
&(\one{x}\succ (\r_\alpha \trl  a))\, ((\r^\alpha \trl \two{x})\succ a') 
\\
&=  \left[ (  \br^\mu \trl S^{-1} ({\one{x}}^*) ) \btrl (\br_\mu \trl (\r_\alpha \trl  a)^*  \right]^* 
 \left[ (  \br^\nu \trl S^{-1} ((\r^\alpha \trl \two{x})^*) ) \btrl (\br_\nu \trl a{'}^* ) \right]^* 
 \\
 &=  \left[ (  \br^\mu \trl S^{-1} ({\one{x}}^*) ) \btrl (\br_\mu \br_\alpha \trl  a^*  \right]^* 
 \left[ (  \br^\nu \trl S^{-1} (\br^\alpha \trl {\two{x}}^*) ) \btrl (\br_\nu \trl a{'}^* ) \right]^* 
\end{align*}
and the two expressions coincide due to $K$-equivariance of  the braided antipode.
 \end{proof}

\subsection{Braided Lie $*$-algebras}
In this section $(K,\r)$ is triangular, with $\r^{*\otimes *}=\r_{21}=\br$.
\begin{defi}
A $*$-structure on a  $K$-braided Lie algebra $\g$ is an antilinear involution $*:\g \to \g$ which satisfies 
\eqref{trlstar} and in addition   $([u,v])^*=[v^*,u^*]$ for all $u,v\in
\g$. 
\end{defi}
The compatibility condition \eqref{trlstar} is well-defined because it is equivalent to require
that the $K$-action $\suc_\g:K\otimes \g\to \g$ in \eqref{succ},  with $(A, \cdot)$ replaced by $(\g, [~,~])$,
is an action of $K$ on the braided Lie algebra  $\g$, which
coincides with the starting action $\trl_\g$.
The proof that $\suc_\g: K\otimes \g\to \g$ is an action on the $K$-module $\g$ is as in
\eqref{trlstar}.
The $K$-equivariance property of the bracket, 
$k\suc_\g[u,v]=[\one{k}\suc_\g u,\two{k}\suc_\g v]$ for all $k\in K$, that is,
$k\!\suc_\g\!\circ\, [~,~]=[~,~]\circ
  k\!\suc_\g$ as maps $\g\otimes \g\to \g$,  is as in
  \eqref{KonAlg}.

  The compatibility of the action $\suc_\g$  with the braided antisymmetry
 $[~,~]=-[~,~]\circ \Psi_\r$ is due to $K$-equivariance of
the braiding  $\Psi_\r$, see \eqref{Psisuc}.
Similarly, the compatibility of the action $\suc_\g$ with the braided Jacobi
  identity (which is an equality between maps obtained from the bracket and
  the braiding) is due to $K$-equivariance of all the maps
  involved.
\medskip

In the present paper the main example of $K$-braided Lie $*$-algebra
is that of braided derivations $\Der{A}$ of a $K$-module $*$-algebra
$A$.
Its $*$-structure is defined by
\beq\label{comp**eql}
\psi^*(a) := -\left( (  \br^\alpha \trl_{\Der{A}} \psi ) (\br_\alpha \trl_A a^*) \right)^* \, ,
\eeq
for all $\psi\in \Der{A}$, $a\in A$. It lifts as an antilinear and antimultiplicative
map to a $*$-structure on the universal enveloping algebra
$L=\U(\Der{A})$. This $*$-structure is compatible with the braided
action $\btrl_A: \U(\Der{A})\otimes A\to A$  defined by
$\psi\btrl_A a=\psi(a)$ for all $\psi\in \Der{A}\subseteq \U(\Der{A})
$, $a\in A$.   Since $\Der{A}$ is the $K$-submodule of
primitive elements,
$S^{-1}(\psi)=-\psi$ and \eqref{comp**eql} also reads 
$
\psi^*(a) = \left( (  \br^\alpha \trl_L S^{-1}(\psi) ) (\br_\alpha \trl_A a^*) \right)^* 
$.
This implies the compatibility:  
\begin{equation}\label{*U}
x^*\btrl_Aa = \left( (  \br^\alpha \trl_L S_L^{-1}(x) )\btrl_A (\br_\alpha \trl_A a^*) \right)^* \,.
\end{equation} This is 
the unique  $*$-structure compatible
with $\btrl_A$, indeed \eqref{comp**} is
equivalent to \eqref{*U}.

If the $K$-module $*$-algebra $A$ is  quasi-commutative, see
\eqref{qcAB}, the $K$-braided Lie $*$-algebra $\Der{A}$ is also a left $A$-module with action $\cdot:
A\otimes \Der{A}\to A$ defined in \eqref{AmodderA}. This  is
compatible with the  $*$-structure: on $A\otimes \Der{A}$ we have  the $*$-structure $(a\otimes
\psi)^*=(\r_\alpha\trl a^*)\otimes (\r^\alpha\trl \psi^*)$ (cf. 
Lemma \ref{lem:starbot}), then
$$(a\cdot \psi)^*=(\r^\alpha\trl a^*)\cdot(\r_\alpha\trl \psi^*)$$ for all
$a\in A$, $\psi\in \Der{A}$, that is,
$*\circ \cdot=\cdot\circ *$ as maps $A\otimes \Der{A}\to \Der{A}$.

\subsection{$*$-structures and twists}

\subsubsection{Twist of Hopf $*$-algebras and their representations}
A twist $\F$ of a Hopf $*$-algebra $K$ is a twist of the Hopf
algebra $K$ that satisfies
\begin{equation}\label{Fstar}
  \F^{*\otimes *}=(S\otimes S)\F_{21}~.
\end{equation}
Then $K_\F$ is a $*$-Hopf algebra with 
\beq\label{Kstar}
*_\F: K\to K \; , \quad k^{*_\F}:=\u \, k^*\bar{\u}
\eeq
 where $\u=\F^\alpha S(\F_\alpha)$ and
$\bar\u=S(\bar\F^\alpha)\bar\F_\alpha$ is its inverse. 
From \eqref{Fstar}  we have $\u^*=\u$, which implies that ${*_\F}$
is  involutive.
The twist condition \eqref{twist} implies the identity $\F\Delta(\u)=\u S(\two{\F_{\!\!\:\beta}})\otimes
\F^\beta S(\one{\F_{\!\!\:\beta}})$ or equivalently
\begin{equation}\label{Fu}
  \F\Delta(\u)=(\u\otimes \u) \bar\F^{*\otimes *}~.
  \end{equation}
  Compatibility with the coproduct then follows:
$\Delta_\F(k^{*_\F})=\Delta_\F(k)^{*_\F\otimes *_\F}$ for all $k\in K$.

If $A$ is a $K$-module $*$-algebra,
$A_\F$ is a $K_\F$-module $*$-algebra with $*:A_\F\to A_F$ that
is the same as the initial $*$ as antilinear map. Indeed this is antimultiplicative:
\begin{equation}\label{aFb*}
(a\cdot_\F b)^*=(\bF_\alpha\trl b)^*(\bF^\alpha\trl
a)^*=S^{-1}({\bF_\alpha}^{*}) \trl
b^*\,S^{-1}({\bF^\alpha}^{\!*})\trl a^*=
\bF^\alpha\trl b^*\, \bF_\alpha\trl a^*=b^*\cdot_\F a^*\,,
\end{equation}
where we used \eqref{trlstar} and \eqref{Fstar}. Moreover it is
compatible with the $K_\F$-action,
\begin{equation}\label{F*action}(k\trl a)^*=(S(k))^*\trl
a^*=(S_\F(k))^{*_\F}\trl a^*
=S^{-1}_\F(k^{*_\F})\trl a^*\,,
\end{equation}
where we used 
$(S_\F(k))^{*_\F}=(\u S(k) \bar \u)^{*_\F}=(S(k))^*$ which holds since $\u^*=\u$.

\subsubsection{Twist of quasitriangular and braided Hopf $*$-algebras, and
  of their
representations}
If $(K, \r, *)$ is  a quasitriangular  Hopf $*$-algebra with $\r$   antireal,  so is
 $(K_\F, \rF,*_\F)$ with $\rF$   antireal. 
From \eqref{Kstar}, \eqref{Fu} and the equivalent expression $\F_{21}^{*\otimes *} (\bar{u} \ot \bar{u})= \Delta^{cop} (\bar{u}) \bF_{21}$, we compute
\begin{align*}
\r_\F^{*_\F\otimes *_\F} &= (\u \ot \u) (\F_{21} \r \bF)^{* \ot *} (\bar{u} \ot \bar{u})
=   \F\Delta(\u)   \r^{* \ot *}   \Delta^{cop} (\bar{u}) \bF_{21} 
= \F\Delta(\u)   \br   \Delta^{cop} (\bar{u}) \bF_{21} 
\\ &= \F    \br  \,   \bF_{21} = \br_\F
\end{align*}
where we used the quasi-cocommutativity property \eqref{iiR}.

For a  $K$-module algebra $L$ one has the  $K_\F$-module
algebra isomorphism
\begin{equation}
  \label{isoFm} \varphi:L_\F\bot_\F L_F\to (L\bot L)_F~,~~x\bot_\F y\mapsto
  \varphi(x\bot_\F y):=\bF^\alpha\trl x \bot  \bF_\alpha\trl y
  \end{equation}
(leading to the monoidal equivalence of the categories of $K$-module algebras
and $K_\F$-module algebras).  When $L$ is a $K$-module $*$-algebra so
is $L\bot L$ while $L_\F$, $(L\bot L)_\F$ and $L_\F\bot_\F L_\F$ are $K_\F$-module $*$-algebras.
This latter 
with $*$-structure (cf. Lemma \ref{lem:starbot})
$(x \bot_\F y)^{*_\F} = \Psi_{\r_\F}(y^*\bot x^*)=({\r_\F}_{\!\alpha} \trl x^{*}) \bot ({\r_F}^{\!\alpha} \trl y^*)$.

\begin{lem}
Let $(K,\r, *)$ be a quasitriangular $*$-Hopf algebra,  with twist
$\F$ and $L$ a
$K$-module $*$-algebra. The isomorphism
$\varphi: L_\F\bot_\F L_\F  \rightarrow (L\bot L)_\F $ in \eqref{isoFm}
is a   $K_\F$-module $*$-algebra isomorphism.
\end{lem}
\begin{proof} We show $\varphi^{-1}\circ *\circ \varphi=*_\F$. Using
the compatibility condition
  \eqref{trlstar}, \eqref{Fstar} and $\r_\F=\F_{21}\r\bF$ we have, for
  all $x,y\in L$, 
\begin{align*}
  \varphi^{-1}((\varphi(x\bot_\F y))^{*})
&=
\varphi^{-1}((\bF^\gamma\trl x\bot\bF_\gamma\trl y)^*)
\\
  &=
    \F^\alpha\r_\beta\trl(\bF^\gamma\trl x)^*\bot_\F\F_\alpha\r^\beta\trl(\bF_\gamma\trl y)^*
\\
&= 
 \F^\alpha\r_\beta S^{-1}({\bF^\gamma}^*)\trl x^*\bot_\F\F_\alpha\r^\beta  S^{-1}(\bF_\gamma^*)\trl y^*
\\
  &= \F^\alpha\r_\beta \bF_\gamma\trl x^*\bot_\F \F_\alpha\r^\beta \bF^\gamma\trl y^*
\\
&=
{\r_\F}_\beta\trl x^*\bot_\F{\r_\F}^\beta\trl y^*
\\
  &=(x\bot_\F y)^{*_\F}\,.
\end{align*}
\\[-3em]
\end{proof}

\begin{thm} $(L_\F,\r_\F, *)$ is a $K_\F$-braided Hopf $*$-algebra.
\end{thm}
\begin{proof}
We already know that $(L_\F,\r_\F)$ is a $K_\F$-braided Hopf algebra,
we prove the compatibility with $*$.  From \eqref{F*action} we see that $(L_\F,*)$ is a $K_\F$-module
 $*$-algebra. Moreover the braided
  coproduct $\Delta_{L_\F}$ is a $*$-algebra map:
 $\Delta_{L_\F}(x^*)={(\Delta_{L_\F}(x))}^{*_\F}$, for all $x\in
 L_\F$. Indeed from \eqref{comod-twist} we have
 $\Delta_{L_\F}\circ *=\varphi^{-1}\circ\Delta_L\circ *=
 \varphi^{-1}\circ
 *\circ\Delta_L=\varphi^{-1}\circ * \circ \varphi \circ \Delta_{L_\F}=*_\F \circ
\Delta_{L_F}$.
Thus  $(L_\F,\r_\F, *)$ is a $K_\F$-braided Hopf $*$-algebra
according to Definition \ref{*HA}.
  \end{proof}

If a $K$-braided Hopf $*$-algebra $L$ acts on a  $K$-module $*$-algebra $A$
then the twisted $K_\F$-braided Hopf $*$-algebra $L_\F$ acts on the
twisted $K_\F$-module
$*$-algebra $A_\F$ with action
$$\btrl_{A_\F}: L_\F\otimes_\F A_\F\to
A_\F~,~~ x\btrl_{A_\F}a=(\bF^\alpha\trl_{L_\F}
x)\btrl_A(\bF_\alpha\trl_{A_\F} a)~.
$$
The compatibility of the $*$-structure of $A_\F$ with  the braided action $\btrl_{A_\F}$ of $L_\F$ reads 
$(x\btrl_{A_\F} a)^{*}=({\br_\F}^\beta\trl_{L_\F}
S^{-1}_{L_\F}(x^*))\btrl_{A_\F}({\br_\F}_\beta\trl_{A_\F} a) $,
cf.  \eqref{comp**}. This
follows from $S^{-1}_L((k\trl x)^*)=S^{-1}_L(S^{-1}(k^*)\trl x^*)=
S^{-1}(k^*)\trl S^{-1}_L(x^*)$, for $k\in K$, $x\in L$,
owing to the $K$-equivariance of the braided antipode $S_L$.

\subsubsection{Twist of braided Lie $*$-algebras of derivations}\label{sec*333}
Let $(K,\r)$ be triangular.
If $\g$ is a $K$-module Lie $*$-algebra then $\g_\F$ is a $\KF$-module Lie
$*$-algebra with the initial involution of $\g$.
The property $([u,v]_\F)^*=[v^*,u^*]_\F$ for all $u,v\in \g_\F$ is
proven along the same lines of those in \eqref{aFb*}.

\medskip
We now consider the $K$-braided Lie $*$-algebra $\g=\Der{A}$
of derivations of the $K$-module $*$-algebra $A$,
with $*$-structure defined
in \eqref{comp**eql} and its universal enveloping $*$-algebra
$L=\U(\Der{A})$.

\begin{prop}\label{D*}
  The isomorphism $\dd: \Der{A}_\F \to  \Der{A_\F}$ of $K_\F$-braided
  Lie algebras of Theorem \ref{thm:Dalg} is a $K$-braided Lie
  $*$-algebra isomorphism. It lifts to the isomorphism $\dd: \U(\Der{A}_\F)
  \to  \U(\Der{A_\F})$ of $K$-braided Hopf $*$-algebras.
  \end{prop}
  \begin{proof}
    We have to show $\dd(\psi^*)(a)={\dd(\psi)}^*(a)$ for all $\psi\in
    \Der{A}_\F$, $a\in A_\F$.
    In analogy with \eqref{aFb*} we have  $ (\bF^\alpha \trl a)^*\otimes (\bF_\alpha \trl \psi)^* 
= (\bF_\alpha \trl a^* ) \otimes (\bF^\alpha \trl \psi^*)$ and
therefore, 
\begin{align*}\dd(\psi^*) (a) 
=  (\bF^\alpha \trl_L \psi^* ) (\bF_\alpha \trl a) 
&= (\bF_\alpha \trl_L \psi )^*  (\bF^\alpha \trl a^*)^*\\
&=[(\r_\beta\trl_L S_L^{-1}(\bF_\alpha\trl_L \psi)) (\r^\beta\bF^\alpha\trl
                                                         a^*)]^* \\
&=[
(\r_\beta \bF_\alpha\trl_L S_L^{-1}( \psi))(\r^\beta\bF^\alpha\trl  a^*) ]^*
\end{align*}
where we used 
 \eqref{comp**eql}  with  $S(\psi)=-\psi$
 (or \eqref{*U}) and $K$-equivariance of the braided antipode $S_L$.
 On the other hand, by definition of the $*$-structure on   $\Der{A_\F}$ we have
 \begin{align*}
{\dd(\psi) }^*(a) =[({\r_\F}_\beta\trl_{L_\F}\!\!\: S_{L_\F}^{-1}(\dd(\psi)))(\r_\F^{\:\beta}\trl  a^*)]^* 
&=
[({\r_\F}_\beta\trl_{L_\F}
                                                                                                          \dd(S_{L}^{-1}(\psi)))(\r_\F^{\:\beta}\trl  a^*)]^* \\
&=[(\dd({\r_\F}_\beta\trl_{L} S_{L}^{-1}(\psi)))(\r_\F^{\:\beta}\trl
                                                                                                                                                                   a^*)]^* \\
&=[( (\bF^\alpha{\r_\F}_\beta\trl_{L} S_{L}^{-1}(\psi)))(\bF_\alpha\r_\F^{\:\beta}\trl
a^*)]^*\,.
\end{align*}
The proof then follows from $\r_\F=\F_{21}\r\F^{-1}$.

The isomorphism $\dd: \U(\Der{A}_\F)
\to  \U(\Der{A_\F})$ of $K$-braided Hopf algebras commutes with the
$*$-structures when restricted to the primitive elements $\Der{A}_\F$
and, since these are the generators,
on all elements of $\U(\Der{A}_\F)$. The
proof  is by induction. If  $x,y\in
\U(\Der{A}_\F)$ satisfy $\dd(x^*)=\dd(x)^*$, $\dd(y^*)=\dd(y)^*$, then
so does their product, $\dd((x\cdot_\F y)^*)=\dd(y^*\cdot_\F x^*)=\dd(y^*)\dotF\dd(x^*)=
\dd(y)^*\dotF\dd(x)^*=(\dd(x)\dotF\dd(y))^*=(\dd(x\cdot_\F y))^*$.
\end{proof}

Finally, for $A$ a quasi-commutative $K$-module $*$-algebra,
$\Der{A}$ is a $K$-braided Lie and $A$-module $*$-algebra. The twisted
algebra $A_\F$ is a
$K_\F$-braided quasi-commutative $*$-algebra and $\Der{A_\F}$ a
$K_\F$-braided Lie and $A_\F$-module $*$-algebra.
In particular,
$$
([\psi,\eta]_{\r_\F})^*=[\eta^*,\psi^*]_{\r_\F}~,\quad (a\dotF
\psi)^*=({\r_{\F}}_\alpha\trl a^*)\dotF ({\r_\F}^\alpha\trl \psi^*)~.
$$
From Theorem \ref{thm:Dalg}, Corollary \ref{DAmodiso} and Proposition
\ref{D*}, 
$\dd: \Der{A}_\F\to  \Der{A_\F}$ is an isomorphism
of
$K_\F$-braided Lie and $A_\F$-module $*$-algebras.

\section{Principal bundles over $S^4_\theta$ and their gauge transformations}\label{sec:examples}

In this section we consider  the twist deformation of the Hopf $SU(2)$-bundle over the  $4$-sphere $S^4_\theta$, and then of the  $SO(4)$-bundle over   $S^4_\theta$, seen as a homogeneous space.

\subsection{The instanton bundle}\label{sec:inst}
The $H=\O(SU(2))$ Hopf--Galois extension $\O(S^4_\theta) \subset \O(S^7_\theta)$ of \cite{LS0} can be
 obtained as a deformation by a twist on $K=\O(\IT^2)$ of the Hopf--Galois extension
 $\O(S^4) \subset \O(S^7)$ of the  classical
 $SU(2)$ Hopf bundle, \cite{ppca}. 
We use that twist deformation in the framework of the  theory developed in \S \ref{sec:TBLA}, 
 to obtain the braided Lie algebras $\DerM{\O(S^7_\theta)}$ and 
 $\mathrm{aut}_{\O(S^4_\theta) }(\O(S^7_\theta) )$ from their classical counterparts 
 $\DerM{\O(S^7}$ and $\mathrm{aut}_{\O(S^4) }(\O(S^7) )$. 

 \subsubsection{The classical Hopf bundle}
Let us start with the Hopf--Galois extension $B \subset A$ of the classical   $SU(2)$-Hopf bundle 
$\pi:  S^7 \to S^4$.
The algebra $A:= \O(S^7)$  is the commutative $*$-algebra of coordinate functions
 on the $7$-sphere $S^7$ with generators
$\{ z_a, z_a^*$,  $a=1,\dots, 4\}$,  satisfying the sphere relation $\sum z^*_a z_a=1$.
It carries a right coation of the Hopf algebra
 $\O(SU(2))$ of coordinate functions on $SU(2)$. This
 is the $*$-algebra generated by commuting elements $\{w_j, w_j^*$, $j=1,2\}$, with $\sum w^*_j w_j=1$, and standard Hopf algebra structure induced 
 from the group structure of $SU(2)$.
The right coaction of $\O(SU(2))$ on  $\O(S^7)$ is defined  on the algebra generators as
\begin{eqnarray}\label{princ-coactSU2}
\delta:\quad  \O(S^7) &\longrightarrow & 
\O(S^7) \ot \O(SU(2))
\\ \nn
\, \mathsf{u}
 &\longmapsto&
\mathsf{u}
\overset{.}{\otimes}
\mathsf{w}
\;, \quad 
\mathsf{u}:=
\begin{pmatrix}
z_1& z_2 & z_3& z_4
\vspace{2pt}
\\
-z_2^*  &
  z_1^* &
  -z_4^*
& z_3^*
\end{pmatrix}^t 
, \quad 
\mathsf{w}:=
\begin{pmatrix}
 w_1 & -w_2^*
\vspace{2pt}
\\
w_2 & w_1^*\end{pmatrix} .
\end{eqnarray}
Here $\overset{.}{\otimes}$ denotes the composition of the tensor product $\ot$ with matrix multiplication.
As usual the coaction is extended to the whole $\O(S^7)$ as a
$*$-algebra morphism.

The subalgebra   
$B=\O(S^7)^{co\O(SU(2))}$ 
of coinvariant elements for the coaction is identified with the algebra $\O(S^4)$ of coordinate  functions  on   the 
 4-sphere $S^4$.
As the algebraic counterpart of the principality of the
 Hopf bundle $\pi:S^7 \to S^4$, one has that
the algebra  $\O(S^7)$ is a (not trival) faithfully flat  Hopf--Galois extension
of $\O(S^4)$.

A set of generators for the algebra $B$ is given by the elements 
\beq\label{4sphere-coinv}
\alpha:= 2(z_1 z_3^* + z^*_2 z_4)~, \quad 
\beta:= 2(z_2 z_3^* - z^*_1 z_4)~, \quad 
x:= z_1 z_1^* + z_2 z_2^* - z_3 z_3^* -z_4 z_4^* 
\eeq 
and  their $*$-conjugated $\alpha^*, \beta^*$, with $x^*=x$. From the $7$-sphere 
relation $\sum z_\mu^* z_\mu=1$, 
 it follows that they satisfy a $4$-sphere relation
$\alpha^* \alpha + \beta^* \beta + x^2=1$. 

\noindent
For future use  we also note these generators satisfy the relations 
\begin{align}\label{regole-grado3}
&   (1-x) z_1 = \alpha z_3  - \beta^* z_4 
&&  (1-x) z_2 = \alpha^* z_4 + \beta z_3  
\nn \\
&(1+x) z_3 = \alpha^* z_1 + \beta^* z_2    
&& (1+x) z_4 = \alpha z_2 - \beta z_1 
\end{align}
together with their $*$-conjugated.

\subsubsection{The equivariant derivations}
Since the sphere $S^7$ and $S^4$ are the homogeneous spaces $S^7=Spin(5)/SU(2)$ and 
$S^4=Spin(5)/Spin(4)\simeq Spin(5)/SU(2)\times SU(2)$, the Hopf fibration 
$S^7\to S^4$ is a $Spin(5)$-equivariant $SU(2)$-principal bundle. Then,  
the right-invariant vector fields $X \in so(5)\simeq
spin(5)$ on  $Spin(5)$  project to the
 right cosets $S^7$ and $S^4$ and generate the $\O(S^7)$-module of
vector fields on $S^7$ and the  $\O(S^4)$-module of those
on $S^4$. 
A convenient generating set for the $\O(S^7)$-module 
 is given by the following right $SU(2)$-invariant vector fields on $S^7$ (cf.~\cite{L06}):
\begin{align}\label{der-sopra1}
&H_1= \tfrac{1}{2}( z_1 \partial_1 - z_1^* \partial^*_1 - z_2 \partial_2  + z_2^* \partial^*_2 - z_3 \partial_3 + z_3^* \partial^*_3 + z_4 \partial_4  - z_4^* \partial^*_4) 
\nn
\\
&H_2 =  \tfrac{1}{2}(- z_1 \partial_1 + z_1^* \partial^*_1 + z_2 \partial_2  - z_2^* \partial^*_2 - z_3 \partial_3 + z_3^* \partial^*_3 + z_4 \partial_4  - z_4^* \partial^*_4)
\end{align}
\begin{align}\label{der-sopra2}
& E_{10} =  \stwo (z_1 \partial_3  - z_3^* \partial^*_1 - z_4 \partial_2 + z_2^* \partial^*_4)
&&  
E_{-10} = \stwo (z_3 \partial_1 - z_1^* \partial^*_3 - z_2 \partial_4  + z_4^* \partial^*_2)
\nn
\\
& E_{01} = \stwo ( z_2 \partial_3  - z_3^* \partial^*_2 + z_4 \partial_1 - z_1^* \partial^*_4) 
&&  
E_{0 -1} = \stwo (z_1 \partial_4  - z_4^* \partial^*_1 + z_3 \partial_2  - z_2^* \partial^*_3)
\nn
\\
& E_{11}= - z_4 \partial_3  + z_3^* \partial^*_4 
&&  
E_{-1-1} = z_4^* \partial^*_3  - z_3 \partial_4 
\nn
\\
& E_{1-1}= - z_1 \partial_2  + z_2^* \partial^*_1  
&&  
E_{-1 1}= - z_2 \partial_1  + z_1^* \partial^*_2 . 
\end{align}
Here the partial derivatives
$\partial_a, \partial_a^*$,  are defined by  $\partial_a(z_c) 
= \delta_{a c}$ and $\partial_a(z_c^*)=0$  and 
similarly  for  $\partial_a^*$, $a, c=1,2,3,4$. The vector fields above are chosen so that their commutators close the Lie algebra 
$so(5)$ in the form
\begin{align}\label{so5}
[H_1,H_2] &=0 \; ; \quad
[H_j,E_\mathsf{r}] = r_j E_\mathsf{r} \; ; \nn \\  
[E_\mathsf{r},E_{-\mathsf{r}}] &= r_1 H_1 + r_2 H_2  \; ; \quad
[E_\mathsf{r}, E_\mathsf{s}]= N_{rs } E_\textsf{r+s} \; .
\end{align}
The elements 
$H_1,H_2$ are the generators of the Cartan subalgebra, 
and $E_\mathsf{r}$ is labelled by  
$$\mathsf{r}=(r_1,r_2)  \in \Gamma= \{(\pm 1 , 0), (0, \pm 1), (\pm1,\pm 1)\}~,$$ one of the eight roots. 
Also,  $N_\textsf{rs}=0$ if $\textsf{r+s}$ is not a root  and $N_\textsf{rs} \in \{1,-1\}$ otherwise.  
The $*$-structure is given by $H_j^*=H_j$ and $E_\mathsf{r}^*=E_{-\mathsf{r}}$, 
The $*$-structure on 
vector fields $X$ is defined by $X^*(f)= (S(X)(f^*))^* = -(X(f^*))^*$ for any function $f$,
and one accordingly checks that for the vector fields in \eqref{der-sopra1} and \eqref{der-sopra2},  $E_{-\mathsf{r}} (z_a) =-(E_{\mathsf{r} }(z^*_a))^*$ and $H_j (z_a)=-(H_j(z^*_a))^*$. 

The vector fields \eqref{der-sopra1} and \eqref{der-sopra2}, being invariant under the action of $SU(2)$, projects to a generating set for 
the  $\O(S^4)$-module of vector fields on $S^4$. Explicitly one finds, 
\begin{align}\label{der-sotto}
& H^\pi_1=  \alpha \partial_\alpha - \alpha^* \partial_{\alpha^*} 
&&  
H^\pi_2 =  \beta \partial_\beta  - \beta^* \partial_{\beta^*}
\nn
\\
& E^\pi_{10} =  \stwo (2x  \partial_{\alpha^*}  - \alpha \partial_x)
&&  
E^\pi_{-10} =  \stwo ( - 2x  \partial_\alpha + \alpha^* \partial_x)
\nn
\\
& E^\pi_{11}= \beta \partial_{\alpha^*}  - \alpha \partial_{\beta^*}
&&  
E^\pi_{-1-1}= -\beta^* \partial_{\alpha}  + \alpha^* \partial_\beta 
\nn
\\
& E^\pi_{01} =  \stwo (2x  \partial_{\beta^*}  - \beta \partial_x)
&&  
E^\pi_{0-1} =  \stwo ( - 2x  \partial_{\beta} + \beta^* \partial_x)
\nn
\\
& E^\pi_{1-1} = \beta^* \partial_{\alpha^*}  - \alpha \partial_{\beta}
&&  
E^\pi_{-11} = - \beta \partial_{\alpha}  + \alpha^* \partial_{\beta^*} 
\end{align}
using analogous partial derivatives on $\O(S^4)$. 
Indeed the $\O(S^4)$-module of vector fields on $S^4$ can be generated 
by the five elements $H_\mu= \partial_{\mu^*} - x_\mu D$, 
for $D = \sum\nolimits_\mu x_\mu \partial _{\mu}$ the Liouville vector field.
The five weights $\mu$ are those of the representation $[5]$ of
$so(5)$ with $$
x_{00} = x  \, , \quad x_{10}=  \stwo \alpha   \, , \quad x_{-10}=  \stwo \alpha^*   
 \, , \quad x_{01}=  \stwo \beta     \, , \quad x_{0-1}=  \stwo \beta^*    \, 
 $$
and sphere relation $\sum_\mu x_\mu^* x_\mu=1$. The commutators $[H_\mu, H_\nu]$ give the generators in \eqref{der-sotto}.

Dually, the vector fields \eqref{der-sopra1} and \eqref{der-sopra2} are $H=\O(SU(2))$-equivariant derivations
and  generate  the 
$\O(S^4)$-module of  such derivations
 \beq\label{der-eq-S7}
\DerM{\O(S^7)} = \{ X\in \Der{\O(S^7)} \, | \, \delta \circ X= (X \ot \id) \circ \delta \} .
\eeq
The general $H$-equivariant derivation is then of the form
\beq
X= b_1 H_1 + b_2 H_2 + \sum\nolimits_\mathsf{r}  b_\mathsf{r}  E_\mathsf{r}   \; 
\eeq
for generic elements $b_j, b_\mathsf{r}  \in \O(S^4)$. 
 These derivations are real, that is $X^*= X$, if and only if $b_j^*= b_j$ and $b_\mathsf{r}^*=b_{-\mathsf{r}}$. 
On  the generators of $\O(S^7)$ the derivation $X$ is given as  
\beq\label{Xder}
X : \O(S^7) \to \O(S^7) \;  , \quad 
\begin{pmatrix}
z_1& z_2 & z_3& z_4
\end{pmatrix}^t  \mapsto \mathsf{M}  \cdot \begin{pmatrix}
z_1& z_2 & z_3& z_4
\end{pmatrix}^t
\eeq
where  $\mathsf{M}$ is the  $4 \times 4$ matrix with entries in $ \O(S^4)$
\beq\label{matrixM-der}
\mathsf{M} = \begin{pmatrix}
a_1 & b_{1-1}^* & - b_{10}^* & b_{01}
\\
-b_{1-1} & -a_1 & -b_{01}^* & - b_{10}
\\
 b_{10} & b_{01} & -a_2 & - b_{11}
\\
- b_{01}^* & b_{10}^* & b_{11}^*& a_2
\end{pmatrix} \; ,
\quad a_1=\tfrac{1}{2} (b_1 -b_2) \, , \, \, a_2=\tfrac{1}{2} (b_1 +b_2) \; .
\eeq
The derivation \eqref{Xder} restricts to 
\begin{align}
X^\pi : \O(S^4) \to \O(S^4) \;  , \quad 
\begin{pmatrix}
\alpha &  \beta & \alpha^* & \beta^* & x
\end{pmatrix}^t
 \mapsto 
\mathsf{M}^\pi
\begin{pmatrix}
\alpha &  \beta & \alpha^* & \beta^* & x
\end{pmatrix}^t
\end{align}
with    
 \beq\label{Msotto}
\mathsf{M}^\pi=
\begin{pmatrix}
b_1  & b_{1-1}^* & 0 & b_{11}^* & \sqrt{2} b_{10}^*
\\
- b_{1-1} &  b_2 & - b_{11}^* & 0 & \sqrt{2} b_{01}^*
\\
0 &  b_{11} & - b_1 &  b_{1-1} & \sqrt{2} b_{10}
\\
- b_{11} & 0 &  - b_{1-1}^* & -  b_2 & \sqrt{2} b_{01}
\\
- b_{10} &  - b_{01} &  - b_{10}^* & -   b_{01}^* &0
\end{pmatrix}.
\eeq

\subsubsection{The Lie algebra of gauge transformations}

We next look for infinitesimal gauge transformations, that is $H$-equivariant derivations $X$ as in \eqref{Xder}
which are vertical: $X^\pi (b)=0$, for  $b \in \O(S^4)$. 
These are the kernel of the matrix $\mathsf{M}^\pi$ in \eqref{Msotto}. 
Their collection $\mathrm{aut}_{\O(S^4)}(\O(S^7))$ is clearly an $\O(S^4)$-module. It is also 
a Lie algebra with Lie bracket $[bX,b'X']=bb'[X,X']$ for any $b,b' \in \O(S^4)$ 
and $X,X' \in \mathrm{aut}_{\O(S^4)}(\O(S^7))$. 

The $Spin(5)$ equivariance of the principal bundle $S^7\to S^4$
implies that the Lie algebra $\mathrm{aut}_{\O(S^4)}(\O(S^7))$ can be organised using 
the representation theory of the Lie algebra $so(5)$. Indeed, the $Spin(5)$ 
action on $S^7$  lifts to  ${\rm{Der}}_{\M^H}(\O(S^7))$  
via the adjoint action, $Ad_gX=L_g\circ X\circ L_g^{-1}$, 
where, as usual, $L_g(a)(p)=a(g^{-1}p)$ for $g\in Spin(5)$, $p\in S^7$ and $a\in\O(S^7)$. 
Since the $Spin(5)$-action closes on
the subalgebra $\O(S^4) \subseteq \O(S^7)$, it also closes
on the Lie subalgebra $\mathrm{aut}_{\O(S^4) }(\O(S^7) )$ of
vertical derivations, indeed   $Ad_gX(b)=L_g( X(L_g^{-1}(b)))=0$ for
all $g\in Spin(5)$, $b\in \O(S^4)$. Infinitesimally, $[T, X](b) = 0$ for
all $T\in so(5)$, $b\in \O(S^4)$.  

It follows that 
$
\mathrm{aut}_{\O(S^4)}(\O(S^7)) =\oplus_\pi V_\pi
$ 
as linear space, 
with the sum over a class of representations $V_\pi$ of $so(5)$ of vertical $\O(SU(2))$-equivariant derivations.
 This decomposition will be worked out in details in \S \ref{se:rtd}.

\begin{prop}\label{3.1}
The Lie algebra  $\mathrm{aut}_{\O(S^4) }(\O(S^7))$ of infinitesimal gauge transformations of the   $\O(SU(2))$-Hopf--Galois extension 
 $\O(S^4 ) \subset \O(S^7 )$ is generated, as an
 $\O(S^4 )$-module, by the elements 
 \begin{align}\label{UWT2}
  K_1&:= 2x H_2 + \beta^* \sqrt{2} E_{01} + \beta \sqrt{2} E_{0-1}  
\nn \\
K_2&: = 2x H_1 + \alpha^* \sqrt{2} E_{10}  + \alpha \sqrt{2} E_{-10} 
\nn \\
 W_{01}&:=  \sqrt{2} \big( \beta H_1 + \alpha^*  E_{11} +  \alpha E_{-1 1} \big)
\nn \\
  W_{0-1}&:=   \sqrt{2} \big( \beta^* H_1 +  \alpha^* E_{1-1} + \alpha E_{-1-1} \big)
\nn \\
 W_{10}&:=  \sqrt{2} \big( \alpha H_2 - \beta^*   E_{11}  + \beta E_{1-1} \big)
\nn \\
  W_{-10}&:=  \sqrt{2} \big( \alpha^*  H_2 + \beta^* E_{-11}  - \beta E_{-1-1} \big)
\nn \\
 W_{11}&:=  2x  E_{11} + \alpha \sqrt{2}E_{01}  - \beta \sqrt{2}E_{10}  
\nn \\
 W_{-1-1}&:= 2x E_{-1-1}  + \alpha^* \sqrt{2}E_{0-1}  - \beta^* \sqrt{2}E_{-10}
\nn \\
W_{1 -1}&:= -2x E_{1-1}  + \beta^* \sqrt{2}E_{10} + \alpha \sqrt{2}E_{0-1}
\nn
\\
W_{-1 1}&:=  -2x E_{-11} + \beta \sqrt{2}E_{-10} + \alpha^* \sqrt{2}E_{01} . 
\end{align}
\end{prop}
\begin{proof}
An $H$-equivariant real derivation $X= b_1 H_1 + b_2 H_2 + \sum\nolimits_\mathsf{r}  b_\mathsf{r}  E_\mathsf{r}$ vanishes on 
$\O(S^4 )$ if  
$\mathsf{M}^\pi\begin{pmatrix} \alpha & \alpha^* & \beta & \beta^* & x \end{pmatrix}^t = 0$, for the associated matrix $\mathsf{M}^\pi$ in \eqref{Msotto}. This reads 
\begin{align}\label{ker}
b_1 \alpha + b_{1-1}^* \beta + b_{11}^* \beta^*+ \sqrt{2} b_{10}^* x &=0
\nn \\
- b_1 \alpha^* +  b_{11} \beta + b_{1-1} \beta^*+ \sqrt{2} b_{10} x&=0
\nn \\
- b_{1-1} \alpha- b_{11}^*\alpha^* + b_2 \beta+  \sqrt{2} b_{01}^* x&=0
\nn \\
- b_{11} \alpha- b_{1-1}^* \alpha^* -  b_2 \beta^*+ \sqrt{2} b_{01} x&=0
\nn \\
 b_{10} \alpha  + b_{10}^* \alpha^* + b_{01} \beta  +   b_{01}^* \beta^*  &=0 \, .
\end{align}
At the algebraic level of the present paper, it is enough to look for
solutions with entries of the matrix $\mathsf{M}^\pi$ that are linear
in the $\O(S^4)$ generators. An explicit computation leads to the
derivations
 \begin{align*}
U_1 &= i \big( 2x H_1 + \alpha^* \sqrt{2} E_{10}  + \alpha \sqrt{2} E_{-10}   \big)
\nn \\
 U_2 &= i \big( 2x H_2 + \beta^* \sqrt{2} E_{01}  + \beta \sqrt{2} E_{0-1} \big)
\nn \\
 W_1 &= (\beta^* - \beta) H_1 + \alpha^* (E_{1-1}- E_{11})+ \alpha (- E_{-1 1}+E_{-1-1} )
\nn \\
 W_2 &= i \big( (\beta^* +\beta) H_1  + \alpha^* (E_{1-1}+ E_{11} )+ \alpha (E_{-1-1}+  E_{-11} )\big) 
\nn \\
 W_3 &= (\alpha^* - \alpha) H_2 + \beta^* (E_{-11} + E_{11} ) - \beta (E_{-1-1}
+ E_{1-1}) 
\nn \\
 W_4 &= i \big( (\alpha^* + \alpha ) H_2 + \beta^* (E_{-11} -  E_{11} ) + \beta (E_{1-1}-E_{-1-1})  \big)
\nn \\
 T_1 &= 2x (E_{11}-  E_{-1-1}) + \sqrt{2} (\alpha E_{01}  - \alpha^* E_{0-1} 
 - \beta E_{10}  + \beta^* E_{-10})
\nn \\
 T_2 &=  i \big(  2x (E_{11}+E_{-1-1}) + \sqrt{2} (\alpha E_{01}  + \alpha^* E_{0-1} 
 - \beta E_{10} - \beta^* E_{-10} )\big)
\nn \\
 T_3 &= 2x (E_{1-1}- E_{-11}) + \sqrt{2} (\beta E_{-10} + \alpha^* E_{01} 
- \beta^* E_{10}   -\alpha E_{0-1})
\nn \\
 T_4 &=  i \big( 2x (E_{1-1}+E_{-11}) - \sqrt{2} (\beta E_{-10} + \alpha^* E_{01} 
 + \beta^* E_{10}  + \alpha E_{0-1} \big) \; .
\end{align*}
The derivations in \eqref{UWT2} are obtained as the linear
combinations
\begin{align*}
&K_1= -i U_2\,,~
K_2= -i U_1\,,
& &W_{01}= - \tfrac{\sqrt{2}}{2}(W_1 + i W_2)\,,
& &W_{0-1}= \tfrac{\sqrt{2}}{2}(W_1 - i W_2)\,,\\
& W_{10}= - \tfrac{\sqrt{2}}{2}(W_3 + i W_4)\,,
& &W_{-10}= \tfrac{\sqrt{2}}{2}(W_3 - i W_4)\,,
& &W_{11}=  \tfrac{1}{2} (T_1 - i T_2)\,,\\
& W_{-1-1}= - \tfrac{1}{2} (T_1 + i T_2)\,,
& &W_{1 -1}= - \tfrac{1}{2} (T_3 - i T_4)\,,
&  &W_{-1 1}=  \tfrac{1}{2} (T_3 + i T_4)\,.
\end{align*}

Each vertical derivation, 
$
X= b_1 H_1 + b_2 H_2 + \sum\nolimits_\mathsf{r}  b_\mathsf{r}  E_\mathsf{r} 
$, 
with  $b_j, b_\mathsf{r}  \in \O(S^4)$ which satisfy  \eqref{ker}   is expressed  as combination of the vertical derivations $K_j, W_\mathsf{r}$ in \eqref{UWT2} as
$$
X=  c_1 K_1 + c_2 K_2 + \sum\nolimits_\mathsf{r}  c_\mathsf{r}  W_\mathsf{r}  
$$
 with coefficients $c_1,c_2 , c_\mathsf{r}  \in \O(S^4)$ given by
\begin{align}\label{c(b)}
&  c_1=  \tfrac{1}{4}  \big( 2 x b_2 +  \sqrt{2}  \beta \ b_{01} +   \sqrt{2} \beta^*  b_{0-1}   \big) 
\quad 
&& c_2 =   \tfrac{1}{4}  \big( 2 x b_1 +  \sqrt{2}  \alpha  b_{10}  +   \sqrt{2} \alpha^*   b_{-10} \big) \nn
 \\
& c_{01}  =  \tfrac{\sqrt{2}}{4}\big( \beta^* b_1 + \alpha  b_{11} +  \alpha^* b_{-1 1} \big)
&&  c_{0-1} =  \tfrac{\sqrt{2}}{4} \big( \beta b_1 +  \alpha b_{1-1} + \alpha^* b_{-1-1} \big)
\nn \\
 & c_{10}  =\tfrac{\sqrt{2}}{4} \big( \alpha^* b_2 - \beta   b_{11}  + \beta^* b_{1-1} \big)
 && c_{-10} =\tfrac{\sqrt{2}}{4} \big( \alpha b_2 + \beta b_{-11}  - \beta^* b_{-1-1} \big)
\nn \\
 & c_{11} =  \tfrac{1}{4}  \big(  2 x  b_{11} + \sqrt{2}   \alpha^*  b_{01}  -  \sqrt{2} \beta^*  b_{10}  \big) 
 && c_{-1-1} = \tfrac{1}{4}  \big(  2 x b_{-1-1}  +  \sqrt{2} \alpha  b_{0-1}  - \sqrt{2}  \beta  b_{-10}\big) 
\nn \\
&c_{1 -1}  =  \tfrac{1}{4}  \big(  -2 x b_{1-1}  +  \sqrt{2}\beta  b_{10} + \sqrt{2} \alpha^*  b_{0-1}\big) 
&& c_{-1 1} =   \tfrac{1}{4}  \big(  -2x  b_{-11} + \sqrt{2} \beta^*  b_{-10} + \sqrt{2}  \alpha b_{01}\big) 
~. 
\end{align}
The proof uses the equation \eqref{ker} for the kernel of $\mathsf{M}^\pi$.
Indeed, from \eqref{UWT2} one computes:
\begin{align*}
 X =& \, c_1 K_1 + c_2 K_2 + \sum\nolimits_\mathsf{r}  c_\mathsf{r}  W_\mathsf{r}  
\\
 =&  \,
  \big(  c_{01}  \sqrt{2}  \beta  + c_2 2x  + c_{0-1}  \sqrt{2}  \beta^* \big) H_1
+ \big( c_1 2x  +   c_{10}  \sqrt{2}  \alpha +  c_{-10}  \sqrt{2} \alpha^*\big)  H_2 
\\ &
+   \big( c_1\beta^* \sqrt{2} +  c_{11} \alpha \sqrt{2}  + c_{-11} \alpha^* \sqrt{2} \big)E_{01} 
+ \big( c_1\beta \sqrt{2}    +  c_{1-1}  \alpha \sqrt{2}  +c_{-1-1}  \alpha^* \sqrt{2} \big)E_{0-1}  
\\ &
+\big( c_2\alpha^* \sqrt{2}   + c_{1-1}  \beta^* \sqrt{2}  -  c_{11} \beta \sqrt{2} \big)E_{10}  
+ \big( c_2 \alpha \sqrt{2}   - c_{-1-1}  \beta^* \sqrt{2}    + c_{-11}  \beta \sqrt{2}\big)E_{-10} 
\\ &
+  \big( c_{01}  \sqrt{2} \alpha^*     - c_{10}  \sqrt{2} \beta^*    + c_{11}  2x \big) E_{11} 
+ \big( c_{01}  \sqrt{2} \alpha   + c_{-10}  \sqrt{2}  \beta^*    - 2xc_{-11}\big)  E_{-11} 
\\ &
+\big( c_{10}  \sqrt{2}  \beta   -  c_{1-1} 2x   +  c_{0-1}  \sqrt{2} \alpha^*\big) E_{1-1} 
+  \big( c_{0-1}  \sqrt{2} \alpha    - c_{-10}  \sqrt{2}  \beta    +  c_{-1-1}  2x \big) E_{-1-1}  
\\
 = & \, b_1 H_1 + b_2 H_2 + \sum\nolimits_\mathsf{r}  b_\mathsf{r}  E_\mathsf{r}
\end{align*}
where the last equality follows from equations \eqref{ker} for the coefficients $b_j$, $b_\mathsf{r}$.
\end{proof}

The action of the vertical derivations $K_j, W_\mathsf{r}$ on the algebra generators
  $z_a$
  of $\O(S^7)$ is listed in Table \ref{table:vfa} in Appendix \ref{app:commut}.

The generators in \eqref{UWT2}  satisfy 
 $K_j (f^*)=-(K_j(f))^*$ and   $W_\mathsf{r} (f^*)  =-(W_{-\mathsf{r} }(f))^*$ 
for $f\in\O(S^7)$,
that is 
$K_j^*= K_j $ and $ W_\mathsf{r}^*= W_\mathsf{-r}$.
This  also follows from   $H_j^*=H_j$ and $E_\mathsf{r}^*=E_{-\mathsf{r}}$ (see page \pageref{so5}) using $(bX)^*= b^* X^*$, for $b \in \O(S^4)$ and $X$ a derivation.

\begin{prop} \label{def-10}
The generators in \eqref{UWT2} transform under the adjoint representation 
of $so(5)$ with highest weight vector  $W_{11}$: 
\begin{align}\label{adjSO}
& H_j \trl K_l = [H_j, K_l]=0 \, , \quad H_j \trl W_\mathsf{r}= [H_j, W_\mathsf{r}] = r_j W_\mathsf{r}  \, , \nn \\
& E_\mathsf{r} \trl K_j = [E_\mathsf{r} , K_j]= - r_j W_\mathsf{r} \, , \nn \\
& E_\mathsf{r} \trl W_\textsf{-r} = [E_\mathsf{r}, W_\textsf{-r}] = r_1 K_1 + r_2 K_2 \, , \quad 
E_\mathsf{r} \trl W_\mathsf{s} = [E_\mathsf{r}, W_\mathsf{s}] = N_{rs} W_\textsf{r+s} \, ,
\end{align}
with $N_\textsf{rs}$ the structure constants of $so(5)$ as before, with $N_\textsf{rs}=0$ if $\textsf{r+s}$ is not a root. 
\end{prop}
\begin{proof}
By direct computation.
\end{proof}
\begin{rem}\label{rem:not-indep}
The generators in \eqref{UWT2} 
are not independent over the algebra $\O(S^4)$. Indeed one finds 
they satisfy the relations: 
\begin{align}\label{relazioni-UWT2} 
\beta   W_{0-1} - \beta^*   W_{01} + \alpha  W_{-10} - \alpha^*   W_{10} & =0 
\nn
\\
 -   \beta  K_2 + \sqrt{2} x   W_{01} - \alpha^*  W_{11} + \alpha  W_{-11}  & =0 
\nn
\\
 - \beta^*  K_2 +   \sqrt{2} x W_{0-1} - \alpha   W_{-1-1} + \alpha^*  W_{1-1}  & =0 
\nn
\\
 -  \alpha  K_1   +  \sqrt{2} x W_{10} + \beta^*  W_{11}  +\beta  W_{1-1}  & =0 
\nn
\\
   -  \alpha^*  K_1  +   \sqrt{2} x W_{-10} + \beta  W_{-1-1} + \beta^*  W_{-11} & =0 \; .
\end{align}
These relations have a deep geometrical meaning. They are the
vanishing `vertical' components of five vector fields which are
horizontal for a canonical connection on the principal bundle \cite{pgc-Atiyah}. 
These horizontal vector fields carry the five dimensional
representation of $so(5)$ the small est not trivial vector representation of $so(5)$ with highest weight vector of weight $(1,0)$. 
On the other hand, any $d$-dimensional representation of  $so(5)$
as vertical vector fields on $S^7$ vanishes when $d<10$. Indeed, 
the only vertical equivariant derivation which is linear in the generators of 
$\O(S^4)$ and with weight $(1,0)$ is $X_{10} =  \alpha^*  H_2 + \beta^* E_{-11} - \beta E_{-1-1}$.
This generator is annihilated by $E_{1,1}$, $E_{1,0}$ and $E_{1,-1}$, but it is not by $E_{0,1}$. 
Since $ E_{0,1} \trl X_{10}$ has weight $(1,1)$ which is not present in the five-dimensional representation,
we conclude that the minimal space of (linear in the generators of $\O(S^4)$) derivations is ten-dimensional. 
\end{rem}

\subsubsection{A representation theoretical decomposition of $\mathrm{aut}_{\O(S^4) }(\O(S^7))$} \label{se:rtd}

The result of multiplying the generators of $\O(S^4)$ with the ten vector fields in \eqref{UWT2} can be organised using the representation theory of $so(5)$  (cf.~\cite{AS1964-I, AS1964-II}). 
An irreducible representation of $so(5)$ is characterised by two non negative integers $(s,n)$ and we 
denote it $[d(s,n)]$. It has highest weight vector of weight $\tfrac{s}{2}(1,1)+ n(1,0)$ and is of  
dimension $d(s,n)= \tfrac{1}{6} (1 + s)(1+n)(2+s+n)(3+s+2n)$. 
The generic vector field in the $\O(S^4)$-module $\mathrm{aut}_{\O(S^4) }(\O(S^7))$ is a combination of 
 the vector fields in \eqref{UWT2} with coefficients given by polynomials in the generators of  $\O(S^4)$.  
The algebra $\O(S^4)$  decomposes in the sum of 
irreducible representations of $so(5)$ (spherical harmonics on $S^4$) as
\beq\label{splittingS4}
 \O(S^4)  =  \bigoplus_{n \in \IN_0} [d(0,n)] 
\eeq
with $[d(0,n)]$ the representation of   highest weight vector $\alpha^{n}$ of weight $(n,0)$ consisting of polynomials of homogeneous degree $n$ in the generators of  $\O(S^4)$ (see Appendix \ref{app:decS4}).

 The 50 vector fields obtained by multiplying the vector fields in \eqref{UWT2} with the generators of   $\O(S^4)$ can be arranged according to the representations 
$[35] \oplus [10] \oplus [5]$.  The highest weight vectors for these three representations are
worked out to be given respectively by:  
\begin{align}\label{hwv-Y}
Z_{21} &= \alpha W_{11}  \, , \nn  \\
Y_{11} &= \sqrt{2} x W_{11} + \alpha W_{01} - \beta W_{10} \, ,  \nn \\
X_{10} &= \beta^* W_{11} + \beta W_{1-1} - \alpha K_1 + \sqrt{2} x W_{10}  ,
\end{align}
with the label denoting the value of the corresponding weight.
\begin{lem} \label{lem:5X10}
When represented as vector fields on the bundle,
the representation $[5]$ generated by the vector $X_{10}$ above vanishes. 
Also, $Y_{11} = - \sqrt{2} W_{11}$ so that the representation $[10]$ generated by 
$Y_{11}$ is the one in Proposition \ref{def-10}. The vector $Z_{21}$ makes up the representation $[35]$,  
none of whose vectors do vanish.
\begin{proof}
The action of $so(5)$ on the vector $X_{10}$ yields the additional four vectors
\begin{align}\label{relazioni-UWT-2} 
X_{00} &=  \beta^*  W_{01} -  \beta W_{0-1} + \alpha^*   W_{10} - \alpha W_{-10} \nn \\
X_{01} &=  \beta K_2  -  \sqrt{2} x W_{01} + \alpha^* W_{11} - \alpha W_{- 11} \nn \\
X_{-10} & = - \alpha^* K_1  + \sqrt{2} x W_{-10} + \beta W_{-1-1} + \beta^* W_{- 11} \nn \\
X_{0-1} & =  - \beta^* K_2 + \sqrt{2}  x W_{0-1} - \alpha W_{-1 -1} + \alpha^* W_{1-1} \,  .
\end{align}
These five derivations vanish on $\O(S^7)$; they are in fact the vanishing  
combinations of derivations in  \eqref{relazioni-UWT2}. (The  
relations   \eqref{relazioni-UWT2} have then a representation-theoretical meaning.)
Using the $4$-sphere relation and the relations in \eqref{regole-grado3} one shows that 
\beq\label{rel-2}
Y_{11} = \sqrt{2} x W_{11} + \alpha W_{01} - \beta W_{10} = - \sqrt{2} W_{11}. 
\eeq
Then the vector $Y_{11}$ generates the starting representation $[10]$ in \eqref{UWT2} as stated.
\end{proof}
\end{lem}

By construction $\mathrm{aut}_{\O(S^4) }(\O(S^7))$ is closed under commutator. 
It turns out that the commutators of the derivations in \eqref{UWT2} can be expressed again in terms of the derivations in \eqref{UWT2} with coefficients which are linear in the generators of $\O(S^4)$ (all commutators are listed in Appendix \ref{app:commut}.
More specifically  we have the following:
\begin{lem}\label{lem2.5}
The commutators of the derivations in  \eqref{UWT2} can be organised according to
the representation $[35] \oplus [10]$ of $so(5)$ already found and generated by $\alpha W_{11}$ and
$W_{11}$. 
\begin{proof}
There are 45 commutators. The non vanishing commutator with highest weight is 
$[W_{11}, W_{10}]$ with weight $(2,1)$. A direct computation shows that
$$
[W_{11}, W_{10}] = - \sqrt{2} \alpha W_{11}
$$
and the corresponding representation is the $[35]$ found in the previous lemma. 

The commutator $[W_{11}, W_{1-1}]$, the only one of weight $(2,0)$, belongs to the representation $[35]$. 
The latter comprises also two combinations of the three vectors
$$
[K_{1}, W_{11}]  , \quad [K_{2}, W_{11}] , \quad [W_{10}, W_{01}] .
$$
of weight $(1,1)$. On the other hand, their combination 
$$
T_{11} = [K_{1}, W_{11}] + [K_{2}, W_{11}] + [W_{10}, W_{01}]
$$
is annihilated by all positive element of $so(5)$ and generates a copy
of the representation $[10]$. In fact this is just the starting
representation in \eqref{UWT2} of  Proposition \ref{def-10}. An explicit computation gives
\begin{align*}
[K_{1}, W_{11}] &= 2 x W_{11} - \sqrt{2} \beta W_{10} \\
[K_{2}, W_{11}] &= 2 x W_{11} + \sqrt{2} \alpha W_{01} \\
[W_{10}, W_{01}] &= - \sqrt{2} \beta W_{10} + \sqrt{2} \alpha W_{01} 
\end{align*}
so that  using the relation \eqref{rel-2} one obtains
\beq\label{rel-3}
T_{11} =  2 \sqrt{2} Y_{11} = - 4 W_{11} .
\eeq
Thus the representation $[10]$ generated by $T_{11}$ is the one in \eqref{UWT2} as stated.
\end{proof}
\end{lem}

\begin{prop} \label{prop:decoinst}
The Lie algebra $\mathrm{aut}_{\O(S^4) }(\O(S^7))$ decomposes as
$$
\mathrm{aut}_{\O(S^4) }(\O(S^7))=\bigoplus\nolimits_{n\in \IN_0} \, [d(2,n)] \, . 
$$
Here $[d(2,n)]$ is the representation of $so(5)$ as derivations on
$\O(S^7)$
of highest weight vector $\alpha^n W_{11}$ of weight $(n+1,1)$
 and dimension 
$
d(2,n)=\tfrac{1}{2}(n+1)(n+4)(2n+5)
$.
\end{prop}

 \begin{proof} 
From the splitting \eqref{splittingS4} of $\O(S^4)$ (and Appendix \ref{app:decS4}), we need to consider the $ 10 \cdot d(0,n)$ 
vector fields obtained by multiplying the 10 vector fields in \eqref{UWT2} with the 
polynomials of degree $n$  in the generators of $\O(S^4)$ that
 are in the representation $[d(0,n)]$ of highest weight vector $\alpha^n$. Of these, $\alpha^n W_{11} =  \alpha^{n-1} Z_{21}$ is a highest weight vector
and generates the representation $[d(2,n)]$.  For the remaining vectors fields, $\alpha^{n-1} Y_{11}$ is a highest weight vector
generating the representation $[d(2,n-1)]$. There is then the highest weight vector  $\alpha^{n-1} X_{10}$ generating the representation $[d(0,n)]$. 
And finally there is the highest weight vector $\alpha^{n-3} \rho^{2}Z_{21}  = \alpha^{n-3}  Z_{21}$ of the representation $[d(2,n-2)]$ (here $\rho^2:=\alpha\alpha^*+\beta \beta^*+x^2=1$). There is no room for additional vectors since a direct computation shows that 
$d(2,n) + d(2,n-1) + d(0,n) + d(2,n-2) = 10 \cdot d(0,n)$.
Since $ Z_{21}=\alpha W_{11}$, 
$Y_{11} =- \sqrt{2} W_{11}$ and $X_{10}=0$,  the only representation which has not yet appeared  in lower degree is $[d(2,n)]$, the one of highest weight $\alpha^n W_{11}$. 
\end{proof}

\subsubsection{Braided derivations and infinitesimal gauge transformations}\label{sec:braidedHopf}

The right invariant vector fields
$H_1$ and $H_2$ of $Spin(5)$  are the vector fields of
a maximal torus $\mathbb{T}^2 \subset Spin(5)$. They define the universal enveloping algebra
$K$ of the abelian Lie algebra $[H_1,H_2]=0$. Their action \eqref{der-sopra1} on  $\O(S^7)$
commutes with the $\O(SU(2))$ right coaction on $\O(S^7)$. 
To the torus 2-cocycle of \cite[Ex. 3.21]{ppca} there corresponds then
a twist 
\beq\label{twist-theta}
\F:= e^{\pi i\theta (H_1 \ot H_2 -H_2 \ot H_1)} \; , \quad \theta \in \IR\; ,
\eeq
with universal $R$-matrix  $\r_\F=\bF^2$. In fact these elements belong to a topological completion of the algebraic tensor
product $K\otimes K$. This fact does not play a role here since we diagonalise 
$\F$ (we systematically use it on eigen-functions of the generators
$H_1, H_2$).

The twist $\F$ in \eqref{twist-theta} hence leads to the $\mathcal{O}(SU(2))$-Hopf--Galois extension $\mathcal{O}(S^4_{\theta})
=\mathcal{O}(S^7_\theta)^{co\O(SU(2))} \subset
\mathcal{O}(S^7_\theta)$ introduced in \cite{LS0}.
 To conform with the literature, 
 in the following we use the subscript $\theta$
  instead of $\F$ for twisted algebras and their multiplications, as
  well as for module structures.
 The algebra $\mathcal{O}(S^7_\theta)$ is generated by coordinates
   $z_a, z_a^*$, $a=1,2,3,4$.
 Their commutation relations are obtained from \eqref{rmod-twist} given that for the action of $H_1$ and $H_2$ the $z_a$ have eigenvalues 
 $\frac{1}{2}(1, -1), \, \frac{1}{2}(-1, 1), \, \frac{1}{2}(-1, -1),
 \, \frac{1}{2}(1, 1)$, for $a=1,2,3,4$.
 The only nontrivial relations among the $z_a$ are:
   $$z_1\dott z_3= e^{\pi  i\theta}z_3\dott z_1~,~~    z_1\dott z_4=e^{-\pi  i\theta}z_4\dott z_1~,~~
      z_2\dott z_3= e^{-\pi  i\theta}z_3\dott z_2~,    ~~z_2\dott z_4= e^{\pi  i\theta}z_4\dott z_2~.$$
 Those with the $z_a^*$ are obtained using that they have
 eigenvalues opposite to the eigenvalues of the $z_a$.
 These coordinates satisfy the relation
 $z_1\dott z^*_1+z_2\dott z^*_2+z_3\dott z^*_3+z_4\dott z^*_4=1$. 
The subalgebra of $\O(SU(2))$-coinvariants is generated by
\begin{align}\label{4sphere-theta}
\alh &:= 2(z_1 \dott z_3^* + z^*_2 \dott z_4), \quad 
\beh:= 2(z_2 \dott z_3^* - z^*_1 \dott z_4), \nn \\
\xh&:= z_1 \dott z_1^* + z_2 \dott z_2^* - z_3 \dott z_3^* -z_4 \dott z_4^* . 
\end{align}
The only nontrivial
  commutation relations are 
  \begin{equation}\label{t4sphererel}
    \alh\dott \beh=e^{-2\pi  i\theta} \beh\dott\!\: \alh~,~~\alh\dott \beh^*=e^{2\pi  i\theta} \beh^*\!\!\:\dott \!\:\alh
 \end{equation}
 and their complex conjugates. They can be
obtained from the twisted multiplication rule
\eqref{rmod-twist} by using that
$\alh, \beh, \xh$ are eigen-functions of $H_1$ and $H_2$,
 with eigenvalues $(1,0)$, $(0,1)$ and $(0,0)$ respectively.
They satisfy the relations 
$\alh\dott \alh^*+\beh\dott \beh^*+\xh\dott  \xh=1$.
From these one then establishes:
\noindent
\begin{align}\label{regole-grado3-theta}
&   (1-\xh) \dott z_1 = \alh \dott z_3 - z_4 \dott \beh^*  \, , 
&&  (1-\xh) \dott z_2 = z_4 \dott \alh^* + \beh \dott z_3  \, , 
\nn \\
&(1+\xh) \dott z_3 = \alh^* \dott z_1 + \beh^* \dott z_2  \, ,
&& (1+\xh) \dott z_4 = z_2 \dott \alh - z_1 \dott \beh  \, .
\end{align}
\begin{rem}\label{rem:cambio-gen}
 The relations \eqref{regole-grado3-theta} are the analogues of the
 classical ones  \eqref{regole-grado3}.
  However, in passing from the algebra $\O(S^4)$ to the algebra $\O(S^4_\theta)
  $ we rescaled by a phase the classical  elements $\alpha, \beta$. In the vector space $\O(S^4_\theta)=\O(S^4) $,  one has 
$\xh=x$ and
\begin{align*}
\alh &= 2(z_1 \dott z_3^* + z^*_2 \dott z_4) = e^{-\frac{\pi  i\theta}{2}}  2(z_1 z_3^* + z^*_2  z_4)=
  e^{-\frac{\pi  i\theta}{2}} \alpha
  \\
\beh &= 2(z_2 \dott z_3^* - z^*_1 \dott z_4) = e^{\frac{\pi  i\theta}{2}}  2(z_2 z_3^* - z^*_1  z_4)
=e^{\frac{\pi i\theta}{2}}\beta~.
\end{align*}
\end{rem}
 
 Since the Lie algebra $so(5)$ 
is a 
braided Lie algebra associated with $K$ with trivial $R$-matrix $\r=1
\ot 1$, we
can twist it to the braided Lie algebra $so_\theta(5)$ associated with
$(K_\F, \r_\F = \bF^2)$. It has Lie brackets  (see Proposition \ref{prop:gf})
\begin{align}\label{so5twist}
[H_1,H_2]_\F &= [H_1,H_2] = 0 \; ; \quad
[H_j,E_\mathsf{r}]_\F= [H_j,E_\mathsf{r}] = r_j E_\mathsf{r} \; ;  \nn \\
[E_\mathsf{r},E_{-\mathsf{r}}]_\F&=[E_\mathsf{r},E_{-\mathsf{r}}] = r_1 H_1 + r_2 H_2  \; ;  \\
[E_\mathsf{r}, E_\mathsf{s}]_\F &= e^{- i \pi \theta {\mathsf{r}} \wedge {\mathsf{s}} } [E_\mathsf{r}, E_\mathsf{s}] 
= e^{- i \pi \theta {\mathsf{r}} \wedge {\mathsf{s}} }N_\textsf{rs} E_\textsf{r+s}\; ,\nn
\end{align}
with ${\mathsf{r}} \wedge {\mathsf{s}}  := r_1 s_2 - r_2 s_1$. Here, as for the $so(5)$-commutators in \eqref{so5},
$N_\textsf{rs}=0$ if $\textsf{r+s}$ is not a root. 

Similarly, the $\O(S^4)$-module and Lie algebra
$
{\rm{Der}}_{\M^H}(\O(S^7))$  is deformed to the
$\O(S^4_\theta)$-module and braided Lie
algebra $({\rm{Der}}_{\M^H}(\O(S^7))_\F , [~,~]_\F, \cdot_\F)$ associated with $(K_\F, \r_\F)$.
This second module is generated 
by  derivations $H_j$ and $E_\mathsf{r}$ with module structure in \eqref{acdotFD}:
$$
a\cdot_\F H_j = a H_j~, \quad a_\mathsf{s}\cdot_\F E_\mathsf{r} = e^{-\pi
  i \theta \, \mathsf{s}\wedge\mathsf{r} } \, a_\mathsf{s} E_\mathsf{r}~,
  $$
 for all $a\in \O(S^7_\theta)$ and $ a_\mathsf{s}\in
 \O(S^7_\theta) $ eigen-functions of $H_j$ with eigenvalues
 ${s}_j$  (being $E_\mathsf{r}$ 
eigenvectors of  $H_j$).
The Lie brackets are determined by those of  $so_\theta(5)$ in
\eqref{so5twist} using equation
\eqref{LieAmodF+} for the module structure.

The Lie algebra of infinitesimal gauge transformations 
 $\mathrm{aut}_{\O(S^4)}(\O(S^7))$  is generated, as an $\O(S^4)$-module, by the operators  in 
\eqref{UWT2}. 
Its twist  deformation is 
the
$\O(S^4_\theta)$-module and braided Lie algebra
$\big(\mathrm{aut}_{\O(S^4)}(\O(S^7) )_\F, [~,~]_\F,\cdot_\F\big)$
associated with $(K_\F, \r_\F= \bF^2)$.
It has braided Lie bracket determined on generators:
\begin{align}\label{aut-twist}
[K_1,K_2]_\F &=[K_1,K_2] \; ; \quad [K_j,W_\mathsf{r}]_\F= [K_j,W_\mathsf{r}] \; ; \nn \\
[W_\mathsf{r}, W_\mathsf{s}]_\F &= e^{- i \pi \theta {\mathsf{r}} \wedge {\mathsf{s}} } [W_\mathsf{r}, W_\mathsf{s}] \; ,
\end{align}
in parallel with the result in \eqref{so5twist}. 
On generic elements 
$X,X'$ in the linear span of the generators in \eqref{UWT2} and $b,b'\in \O(S^4)$, 
the equation \eqref{LieAmodF+} gives 
 \begin{equation}   
[b \cdot_\F X, b'\cdot_\F X']_\F\,= 
    b\cdot_\F ({\r_\F}_\alpha\trl b')  \cdot_\F[{\r_\F}^\alh\trl
    X,X']_{\F}~.   
  \end{equation}

From Corollary \ref{DAmodiso} we have that
${\rm{Der}}_{\M^H}(\O(S^7_\theta)) = \dd(({\rm{Der}}_{\M^H}(\O(S^7))_\F)$. Thus: 
\begin{prop}\label{derivationsinstanton}
The braided Lie algebra ${\rm{Der}}_{\M^H}(\O(S^7_\theta))$ of equivariant derivations 
of the  $\O(SU(2))$-Hopf--Galois extension  $\O(S^4_\theta) \subset \O(S^7_\theta)$
is generated, as an  $\O(S^4_\theta)$-module by elements:
\begin{equation}\label{DgenerDer}
\widetilde{H}_j :=\dd (H_j) \, , \quad  \widetilde{E}_\mathsf{r} := \dd(E_\mathsf{r})\;
, \quad j=1,2 \, , \quad \mathsf{r} \in \Gamma
\end{equation}
with bracket closing the braided Lie algebra $so_\theta(5)$, 
\begin{align}\label{sotheta5} 
&[\widetilde{H}_1,\widetilde{H}_2]_\rF =\dd([{H}_1,{H}_2]) = 0 \; ;  
&&
[\widetilde{H}_j,\widetilde{E}_\mathsf{r}]_\rF = \dd([ {H}_j,  {E}_\mathsf{r}]) = r_j \widetilde{E}_\mathsf{r}\; ; 
 \\
&[\widetilde{E}_\mathsf{r}, \widetilde{E}_{-\mathsf{r}}]_\rF = \dd([ {E}_\mathsf{r},  {E}_{-\mathsf{r}}])
= \sum\nolimits_j  r_j \widetilde{H}_j  ; 
&&
[\widetilde{E}_\mathsf{r}, \widetilde{E}_\mathsf{s}]_\rF = e^{- i \pi \theta {\mathsf{r}} \wedge {\mathsf{s}} } \, 
\dd([ {E}_\mathsf{r}, {E}_\mathsf{s}]) = e^{- i \pi \theta {\mathsf{r}} \wedge {\mathsf{s}} } N_\textsf{rs} \widetilde{E}_\textsf{r+s}\; \nonumber
\end{align}
(and $N_\textsf{rs}=0$ if $\textsf{r+s}$ is not a root). The
$\O(S^4_\theta)$-module structure is in \eqref{LieAFmod} (with $\dotF =\dott$). 
\end{prop}
The braided Lie algebras in \eqref{so5twist} and in \eqref{sotheta5}
are isomorphic via the isomorphism $\dd$.
For $a_\mathsf{s} \in \O(S^7_\theta)$ an eigen-function of $H_j$ of eigenvalue ${s}_j$, the derivation  
$\widetilde{E}_\mathsf{r}$
acts as
\beq\label{tactder}
\widetilde{E}_\mathsf{r} (a_\mathsf{s}) 
= (\bF^\alpha \trl {E}_\mathsf{r} )(\bF_\alpha \trl a_\mathsf{s}) 
= e^{- i \pi \theta {\mathsf{r}} \wedge {\mathsf{s}}} {E}_\mathsf{r} (a_\mathsf{s}) \, .
\eeq
On the product of two such eigen-functions  $a_\mathsf{s}, {a}_\textsf{m}  \in \O(S^7_\theta)$, we can explicitly see that $\widetilde{E}_\mathsf{r} $ acts as a braided derivation, with respect to the braiding  $\rF = \F_{21} \, \bF = \bF^2$:
\begin{align} \label{tactderder}
\widetilde{E}_\mathsf{r} (a_\mathsf{s} \dott {a}_\textsf{m} ) 
& = e^{- i \pi \theta \wdg{r}{(s+m)} } e^{- i \pi \theta \wdg{s}{m}}  {E}_\mathsf{r} (a_\mathsf{s}   {a}_\textsf{m} ) 
\nn \\
& = e^{- i \pi \theta (\wdg{r}{(s+m)}+ \wdg{s}{m}) } 
\big[ {E}_\mathsf{r} (a_\mathsf{s})   {a}_\textsf{m}  + a_\mathsf{s} {E}_\mathsf{r} ({a}_\textsf{m} ) \big]
\nn \\
& = e^{- i \pi \theta (\wdg{r}{(s+m)}+ \wdg{s}{m}) } 
\big[ e^{i \pi \theta \wdg{(r+s)}{m} }  {E}_\mathsf{r} (a_\mathsf{s}) \dott {a}_\textsf{m}
+
e^{ i \pi \theta \wdg{s}{(r+m)}   } 
 a_\mathsf{s}) \dott {E}_\mathsf{r} ({a}_\textsf{m} ) \big]
\nn \\
& = e^{- i \pi \theta {\mathsf{r}} \wedge {\mathsf{s}} }  {E}_\mathsf{r} (a_\mathsf{s}) \dott {a}_\textsf{m}
+
e^{- 2 i \pi \theta  ({\mathsf{r}} \wedge {\mathsf{s}}+ \wdg{r}{m})       } 
 a_\mathsf{s} \dott {E}_\mathsf{r} ({a}_\textsf{m} )  
\nn \\
& =  \widetilde{E}_\mathsf{r} (a_\mathsf{s}) \dott {a}_\textsf{m}
+
e^{-2  i \pi \theta  {\mathsf{r}} \wedge {\mathsf{s}}      } 
 a_\mathsf{s} \dott \widetilde{E}_\mathsf{r} ({a}_\textsf{m} ) .
\end{align}

Using these results for the subalgebra $\mathrm{aut}_{\mathcal{O}(S^4_\theta)
}(\mathcal{O}(S^7_\theta))=\dd(\mathrm{aut}_{\mathcal{O}(S^4)
}(\mathcal{O}(S^7) )_\F)$ of vertical derivations, we have the following characterization of 
$\mathrm{aut}_{\mathcal{O}(S^4_\theta)}(\mathcal{O}(S^7_\theta))$.

\begin{prop}\label{gaugeinstanton}
The braided Lie algebra $\mathrm{aut}_{\O(S^4_\theta) }(\O(S^7_\theta) )$ of infinitesimal gauge 
transformations of the  $\O(SU(2))$-Hopf--Galois extension 
$\O(S^4_\theta) \subset \O(S^7_\theta)$
is generated, as an  $\O(S^4_\theta)$-module, by the elements
\begin{equation}\label{Dgenerators}
\widetilde{K}_j :=\dd (K_j) \, , \quad  \widetilde{W}_\mathsf{r} := \dd(W_\mathsf{r})\;
, \quad j=1,2 \, , \quad \mathsf{r} \in \Gamma
\end{equation}
with bracket given in Table \ref{table:so5theta}.
The braided Lie bracket of generic elements $\widetilde{X},\widetilde{X}'$ in $\mathrm{aut}_{\O(S^4_\theta)}(\O(S^7_\theta) )$ 
and $b,b'\in \O(S^4_\theta)$ 
is then given by
\begin{equation}\label{Lierelations}
    [b \dott \widetilde{X}, b' \dott \widetilde{X}']_{\r_\F}\,=
    b \dott (\rF_\alpha\trl b')  \dott [\rF^\alpha\trl
    \widetilde{X},\widetilde{X}']_{\r_\F}~ .  
  \end{equation}
\end{prop}
\begin{proof} 
For all $\widetilde{X},\widetilde{X}' \in \mathrm{aut}_{\O(S^4_\theta) }(\O(S^7_\theta) )$
 we have
$[\widetilde{X}, \widetilde{X}']_{\r_\F}=\dd([X,X']_\F)$ 
from Proposition \ref{autautF} with the bracket on the right hand side given 
in \eqref{aut-twist}. 
Using the classical Lie brackets 
 listed in Appendix \ref{app:commut},
we can compute the brackets of the generators 
 of the braided gauge Lie algebra
 $\mathrm{aut}_{\mathcal{O}(S^4_\theta) }(\mathcal{O}(S^7_\theta))$. 
For instance, for $[\widetilde W_{-1-1},\widetilde W_{01}]_{\r_\F} $ we first compute 
\begin{align*}
[W_{-1-1}, W_{01}]_\F & = e^{\pi i\theta} [W_{-1-1},W_{01}] 
=  e^{\pi i\theta}(   \sqrt{2} \beta W_{-1-1} - \sqrt{2} \alpha^*(K_1+K_2 ) )
\\ 
& = e^{\pi i\theta}(     e^{\pi i\theta}\sqrt{2} \beta \cdot_\F W_{-1-1} - \sqrt{2} \alpha^*\cdot_\F(K_1+K_2 )).
\end{align*}
Here we used \eqref{aut-twist} to relate the brackets $[~,~]_\F$ and $[~,~]$, the module structure of
$\mathrm{aut}_{\mathcal{O}(S^4) }(\mathcal{O}(S^7) )_\F$ in 
 \eqref{acdotFD} and that the coordinates of the sphere $S^4$
are eigen-functions of $H_1$ and $H_2$. Next, applying the algebra map $\dd$ leads to
\begin{align*}
[\widetilde W_{-1-1},\widetilde W_{01}]_{\r_\F} 
&= \dd ([W_{-1-1}, W_{01}]_\F)
= 
e^{2\pi i\theta}\sqrt{2} \phib \beh \dott \widetilde W_{-1-1} -  e^{\pi i\theta}\sqrt{2}
\phia^* \alh^* \, \dott(\widetilde K_1+\widetilde K_2 ) .
\end{align*}
Here to pass from the generators of $\mathcal{O}(S^4)$ to those of $\mathcal{O}(S^4_\theta)$ 
we used (see Remark \ref{rem:cambio-gen}):
$$
\alpha= \phia \alh:=e^{\frac{\pi i\theta}{2}}\alh
~, ~~\beta=\phib\beh:=e^{-\frac{\pi i\theta}{2}}\beh \, .
$$
Half of the brackets among the generators \eqref{Dgenerators} of the braided Lie
algebra   $\mathrm{aut}_{\mathcal{O}(S^4_\theta) }(\mathcal{O}(S^7_\theta))$ are  listed 
 in Table  \ref{table:so5theta}. 
Using the $*$-structure the remaining ones can be obtained
   from
$$
[{{\widetilde X}'}{}^*, {\widetilde X}^*]_{\r_\F}=([\widetilde X, \widetilde X']_{\r_\F})^*  , \qquad (b\dott\widetilde X)^*=
({\r_\F}_\gamma \trl b^*) \dott ({\r_\F}^\gamma \trl \widetilde X^*)\, ,
$$
see  \S \ref{sec*333}.
Then, starting from $\widetilde{K}_j^*=\widetilde{K}_j$ and $\widetilde{W}^*_\mathsf{r}=  \widetilde{W}_{-\mathsf{r}}$ (cf. page \pageref{UWT2}):
$$
[\widetilde{K}_{j} , \widetilde{W}_{-\mathsf{r}}]_{\r_\F}
= [\widetilde{K}_j^*, \widetilde{W}_\mathsf{r}^*]_{\r_\F}
=([\widetilde{W}_\mathsf{r} , \widetilde{K}_j]_{\r_\F})^* 
\; , \quad
[\widetilde{W}_{- \mathsf{r}} , \widetilde{W}_{-\mathsf{s}}]_{\r_\F}
= [\widetilde{W}_\mathsf{r}^*, \widetilde{W}_\mathsf{s}^*]_{\r_\F}
=([\widetilde{W}_\mathsf{s} , \widetilde{W}_\mathsf{r}]_{\r_\F})^*  \; ,
$$
with
$$
[\widetilde{W}_\mathsf{r} , \widetilde{K}_j]_{\r_\F}= - [\widetilde{K}_j , \widetilde{W}_\mathsf{r}]_{\r_\F}
\quad , \qquad
[\widetilde{W}_\mathsf{s} , \widetilde{W}_\mathsf{r}]_{\r_\F}
= - e^{ -2i \pi \theta {\mathsf{s}} \wedge {\mathsf{r}} } [\widetilde{W}_\mathsf{r} , \widetilde{W}_\mathsf{s}]_{\r_\F}
\, .
$$
For instance,
\begin{align*}
 [\widetilde{K}_2, \widetilde{W}_{0-1}]_{\r_\F}
  =  -( [\widetilde{K}_2, \widetilde{W}_{01}]_{\r_\F})^*
 & = -  \sqrt{2} (e^{-\pi i\theta} \phia^* \alh^*   \dott \widetilde{W}_{11} 
+   e^{\pi i\theta}\phia\alh  \dott \widetilde{W}_{-11})^*
\\
 & = -  \sqrt{2} (e^{- \pi i\theta}  \phia \alh\, \dott \widetilde{W}_{-1-1}+ e^{\pi i\theta} \phia^* \alh^*  \dott \widetilde{W}_{1-1}  )
 \end{align*}
while 
\begin{align*} 
 [\widetilde{W}_{10}, \widetilde{W}_{0-1}]_{\r_\F}
 = - ( e^{- 2\pi i\theta} [\widetilde{W}_{-10}, \widetilde{W}_{01}]_{\r_\F})^*
 &= - e^{ 2\pi i\theta} \sqrt{2} (e^{2\pi i\theta}\phib \beh \dott \widetilde{W}_{-10} +  \phia^* \alh^*   \dott \widetilde{W}_{01})^*
 \\
 & = -  e^{ 2\pi i\theta}  \sqrt{2} (  \phib^* \beh^* \dott \widetilde{W}_{10} +e^{- 2\pi i\theta} \phia \alh \dott \widetilde{W}_{0-1} )
 \\
 & = -   \sqrt{2} ( e^{ 2\pi i\theta}   \phib^* \beh^* \dott \widetilde{W}_{10} +  \phia \alh \dott \widetilde{W}_{0-1} )
 \; .
 \end{align*}
\end{proof}
The action of any element $\widetilde{W}_\mathsf{r}$ on an algebra element 
$a_\mathsf{s} \in \O(S^7_\theta)$ is as in \eqref{tactder}, 
\beq\label{tactderv}
\widetilde{W}_\mathsf{r} (a_\mathsf{s}) = e^{- i \pi \theta {\mathsf{r}} \wedge {\mathsf{s}}} {W}_\mathsf{r} (a_\mathsf{s}) \, ,
\eeq
with a braided derivation property as in \eqref{tactderder},  
\beq \label{tactderderv}
\widetilde{W}_\mathsf{r} (a_\mathsf{s} \dott {a}_\textsf{m} ) 
=  \widetilde{W}_\mathsf{r} (a_\mathsf{s}) \dott {a}_\textsf{m}
+ e^{-2  i \pi \theta  {\mathsf{r}} \wedge {\mathsf{s}} } a_\mathsf{s} \dott \widetilde{W}_\mathsf{r} ({a}_\textsf{m} ) .
\eeq

\begin{table}[H]\caption{The braided brackets of vertical derivations}
 \label{table:so5theta}
 \begin{align*}
& [\widetilde K_1, \widetilde K_2]_{\r_\F} 
= \sqrt{2} ( \phia^* \alh^*  \!\dott\widetilde W_{10} - \phia\alh \!\:\dott\widetilde W_{-10}) 
\nn \\
&[\widetilde K_1, \widetilde W_{01}]_{\r_\F} = - \sqrt{2} \phib \beh\dott\widetilde K_2 + 2 x    \dott\widetilde W_{01} 
\nn \\
&[\widetilde K_1, \widetilde W_{1-1}]_{\r_\F} =- 2 x \dott\widetilde W_{1-1} + \sqrt{2}e^{\pi i\theta} \phib^* \beh^* \! \dott\widetilde W_{10}\nn \\
&[\widetilde K_1, \widetilde W_{10}]_{\r_\F} =\sqrt{2}e^{-\pi i\theta}\phib \beh \dott\widetilde W_{1-1}-\sqrt{2}e^{\pi i\theta} \phib^* \beh^* \!  \dott\widetilde W_{11} \nn \\
&[\widetilde K_1, \widetilde W_{11}]_{\r_\F} = 2 x                                                                                                                                                                                  \dott\widetilde W_{11} - \sqrt{2} e^{-\pi i\theta}\phib \beh \dott \widetilde W_{10} \nn \\
&
[\widetilde K_2, \widetilde W_{01}]_{\r_\F} = \sqrt{2}e^{-\pi i\theta} \phia^* \alh^* \!  \dott\widetilde W_{11} 
+ \sqrt{2}e^{\pi i\theta}\phia\alh \!\:\dott\widetilde W_{-11}\nn \\
 &  [\widetilde K_2, \widetilde W_{1-1}]_{\r_\F}=2 x
   \dott\widetilde W_{1-1} - \sqrt{2}e^{-\pi i\theta}  \phia\alh\!\:
   \dott\widetilde W_{0-1} \nn \\  
&[\widetilde K_2, \widetilde W_{10}]_{\r_\F} = {2}  x
    \dott\widetilde W_{10} -\sqrt{2} \phia\alh\!\:\dott    \widetilde K_1 \nn \\
   & [\widetilde K_2, \widetilde W_{11}]_{\r_\F}=2 x
   \dott\widetilde W_{11} + \sqrt{2} e^{\pi i\theta}\phia\alh\:\!
   \dott\widetilde W_{01} 
\end{align*}
\begin{align*}
  &
[\widetilde W_{01}, \widetilde W_{1-1}]_{\r_\F}= \sqrt{2} \phib \beh \dott\widetilde W_{1-1} + \sqrt{2} e^{\pi i\theta}\phia\alh \!\:\dott(\widetilde K_2 -\widetilde K_1 )
                              \nn \\
&[\widetilde W_{01}, \widetilde W_{10}]_{\r_\F}=  \sqrt{2} \phib \beh                            \dott\widetilde W_{10} - \sqrt{2}e^{\pi i\theta}\phia\alh \!\:\dott\widetilde W_{01}     \nn \\
&[\widetilde W_{01}, \widetilde W_{11}]_{\r_\F}=  \sqrt{2} \phib \beh \dott\widetilde W_{11}\nn \\
&[\widetilde W_{1-1}, \widetilde W_{10} ]_{\r_\F}=     \sqrt{2}
      \phia\alh \!\:\dott\widetilde W_{1-1}\nn \\
  & [\widetilde W_{1-1}, \widetilde W_{11}]_{\r_\F}=  -\sqrt{2}
     e^{-\pi i\theta}  \phia\alh \!\:\dott\widetilde W_{10}\nn \\
 &[\widetilde W_{10}, \widetilde W_{11}]_{\r_\F}=   \sqrt{2}   \phia\alh\:\!
    \dott\widetilde W_{11}  
\nn \\
&[ \widetilde W_{-1-1},\widetilde W_{01}]_{\r_\F}=     \sqrt{2} e^{2\pi i\theta}\phib \beh \dott\widetilde W_{-1-1} - \sqrt{2}e^{\pi i\theta} \phia^* \alh^* \!\dott(\widetilde K_1 +\widetilde K_2 )\nn \\
&[\widetilde W_{-1-1}, \widetilde W_{1-1}]_{\r_\F} =  \sqrt{2}   e^{-2\pi i\theta}\phib^* \beh^* \!\dott\widetilde W_{0-1}\nn \\
&[\widetilde W_{-1-1}, \widetilde W_{10}]_{\r_\F}=  \sqrt{2}   \phia\alh\:\!
                                                                                                                         \dott\widetilde W_{-1-1}  + \sqrt{2} e^{\pi i\theta}\phib^* \beh^* \! \dott (\widetilde K_1 +\widetilde K_2 )\nn \\
&[ \widetilde W_{-1-1}, \widetilde W_{11}]_{\r_\F}= -2x \dott (\widetilde K_1 +\widetilde K_2 ) - \sqrt{2}   \phia\alh \!\:\dott\widetilde W_{-10}- \sqrt{2}   \phib \beh \dott\widetilde W_{0-1}\nn \\
&[\widetilde W_{-10},\widetilde W_{01}]_{\r_\F}=  \sqrt{2} e^{2\pi i\theta}\phib \beh \dott\widetilde W_{-10} + \sqrt{2}\phia^* \alh^* \! \dott\widetilde W_{01}  \nn \\
&[\widetilde W_{-10}, \widetilde W_{1-1}]_{\r_\F}=  -  \sqrt{2}
       \phia^* \alh^* \! \dott\widetilde W_{1-1}  +  \sqrt{2} e^{-\pi i\theta}\phib^* \beh^* \! \dott (\widetilde K_2 -\widetilde K_1 )\nn \\
& [\widetilde W_{10}, \widetilde W_{-10}]_{\r_\F}= \sqrt{2} (\phib^* \beh^* \! \dott\widetilde W_{01} + \phib \beh \dott\widetilde W_{0-1} )\nn \\
&[ \widetilde W_{-11},\widetilde W_{01}]_{\r_\F}=  \sqrt{2} e^{2\pi i\theta}\phib \beh \dott\widetilde W_{-11}\nn \\
 &[\widetilde W_{0-1}, \widetilde W_{01}]_{\r_\F}= \sqrt{2} ( \phia^* \alh^* \!
   \dott\widetilde W_{10} - \phia\alh \!\:\dott\widetilde W_{-1 0})\nn \\
& [\widetilde W_{-11},\widetilde W_{1-1}]_{\r_\F} = 2x\dott (
            \widetilde K_1 -\widetilde K_2 ) - \sqrt{2}   \phib^* \beh^* \! \dott\widetilde W_{01}  +  \sqrt{2}   \phia\alh \!\:\dott\widetilde W_{-10}   
 \end{align*}
 \end{table}


\subsection{The orthogonal bundle on the homogeneous space $S^{4}_\theta$}\label{S2n} 

 The $4$-sphere of the previous example is the prototype of more general 
noncommutative $\theta$-spheres $S^{2n}_\theta$. These are quantum
homogeneous spaces of quantum groups $SO_\theta(2n+1,\mathbb{R})$
\cite{var, cdv}.
The Hopf--Galois extension $\O(S^{2n}_\theta)=\O(SO_\theta(2n+1,\mathbb{R}))^{co \O(SO_\theta(2n,\mathbb{R}))} \subset
\O(SO_\theta(2n+1,\mathbb{R}))$  was obtained in \cite{ppca} from the extension of the classical $SO(2n)$-bundle $SO(2n+1) \to S^{2n}$  via a 
twist deformation process for quantum homogeneous spaces.
 
 As in the previous section, the modules of infinitesimal gauge transformations of these noncommutative Hopf--Galois extensions 
 are obtained  by deforming those of the corresponding classical bundles. 
We study here the case $n=2$ of the Hopf--Galois extension  
$\O(S^{4}_\theta)=\O(SO_\theta(5,\mathbb{R}))^{co \O(SO_\theta(4,\mathbb{R}))} \subset
\O(SO_\theta(5,\mathbb{R}))$. We address the general case in \cite{pgc-Atiyah}.

Let $\O(M(4, \IR))$ be the  commutative complex $*$-algebra  
with generators $m_{JL}$, and capital indices $J,L$ running from $1$ to $4$. 
It has the standard bialgebra structures
\begin{equation}
\Delta(M)=M \overset{.}{\otimes} M \quad , \quad \varepsilon(M)=\II_4, \quad \mbox{ for } \quad M=(m_{JL}), 
\end{equation}
 in matrix notation,
where $\overset{.}{\otimes}$ denotes the combination of tensor product and matrix multiplication. 
The Hopf algebra $\O(SO(4, \IR))  $ of
coordinate functions on $SO(4, \IR)$ is the quotient 
 of $\O(M(4, \IR))$ by the 
  bialgebra ideal 
\begin{equation}\label{idealQ}
I_Q:= \langle\, M^t Q M -Q \; ; \; M Q M^t-Q\; 
\,\rangle \; 
, \quad Q:=\left( \begin{smallmatrix}
0 &\II_2 
\\
 \II_2 & 0
\end{smallmatrix} \right) = Q^t=Q^{-1}~
\end{equation}
and the further assumption $\det(M)=1$.
Indeed, this is a $*$-ideal for  the $*$-structure $*(M)=QMQ^t$ in $\O(M(4, \IR))$. 
If we introduce on the set  of indices $\{1, \dots, 4\}$ the involution defined by $1'=3$ and $2'=4$,
the $*$-structure can be simply described as 
$$
(m_{JL})^* = m_{J'L'} , \qquad J,L = 1, \cdots, 4 .
$$
The  $*$-bialgebra $\O(SO(4, \IR))$ is a $*$-Hopf algebra with antipode $S(M):= Q M^tQ^{-1}$.

Similarly, one has the algebra  $\O(M(5, \IR))$, the  commutative $*$-bialgebra with generators 
$n_{JL}$ (capital indices $J,L$ now run from $1$ to $5$).
The coproduct and counit are  
\begin{equation}
\Delta(N)=N \overset{.}{\otimes} N , \quad \varepsilon(N)=\II_5 , \quad \mbox{ for } \quad
N:= ( n_{JL} ) .
\end{equation}
The algebra of coordinate functions on $SO(5, \IR)$  is  the quotient of $\O(M(5, \IR))$ by the bialgebra $*$-ideal 
\begin{equation}\label{idealQ+1}
J_Q= \langle N^t Q N -Q \; ; \; N Q N^t-Q 
\;
\rangle \; , \quad Q:= \left( \begin{smallmatrix}
0 & \II_2 &0
\\
\II_2 & 0 &0
\\
0 & 0 & 1
\end{smallmatrix} \right),
\end{equation}
and the additional requirement that  $\det(N)=1$. 
The $*$-structure is now 
$$
(n_{JL})^* = n_{J'L'} , \qquad J,L = 1, \cdots, 5 
$$
with $5'=5$. 
Then $\O(SO(5, \IR))$ a $*$-Hopf algebra with
antipode $S(N):= Q N^tQ^{-1}$. 

We shall select the last column of $N$ by writing $n_{j5} = u_j$, for $j=1, \cdots, 4$ and $n_{55} = x$.

The  surjective Hopf $*$-algebra morphism
\beq\label{pi}
\pi:\O(SO(5, \IR)) \longrightarrow \O(SO(4, \IR)) \; , \quad 
N
\longmapsto
\left( \begin{smallmatrix}
M &0 
\\
0 &  1
\end{smallmatrix} \right)
\eeq
induces a  right coaction of  $\O(SO(4, \IR))$ on $\O(SO(5, \IR))$:  
\begin{align}\label{coactionSO(4)}
\delta:= (\id \ot \pi)\Delta: \O(SO(5, \IR)) &\longrightarrow \O(SO(5, \IR)) \ot \O(SO(4, \IR)) 
\nn \\
N & \longmapsto N \overset{.}{\otimes} 
\begin{pmatrix}
M &0 
\\
0 &  1
\end{pmatrix} .
\end{align}

\noindent
The subalgebra of $  \O(SO(5, \IR))$ made of coinvariant elements 
is isomorphic to the algebra of coordinate functions $\O(S^{4})$
on the $4$-sphere $S^{4} $.  It is indeed generated by  the elements  $u_i, u_i^*$ and $x$ in the last column of the defining matrix $N=(n_{JK})$ of $\O(SO(5, \IR))$, which  satisfy the 
sphere equation
$ (S(N)N)_{55}=   2 (u_1^* u_1 + u_2^* u_2) +x^2=1$.
The algebra extension $\O(S^{4})=\O(SO(5,\mathbb{R}))^{co \O(SO(4,\mathbb{R}))} \subset
\O(SO(5,\mathbb{R}))$ is Hopf--Galois.
 
The coordinate algebra of this orthogonal 4-sphere is isomorphic to the one of the 4-sphere of the previous section (the base space algebra of the $SU(2)$-fibration) with generators  in \eqref{4sphere-coinv}   via the identification  
\beq\label{iso4}
 u_1 \to \tfrac{1}{\sqrt{2}}  \alpha, \quad u_2 \to \tfrac{1}{\sqrt{2}}  \beta,  \quad x \to x.
\eeq  

\medskip

We now determine infinitesimal gauge transformations of the Hopf--Galois extension  
$\O(S^{4}) \subset \O(SO(5,\mathbb{R}))$ using, as done for the previous example, the representation theory 
of $so(5)$ as vector fields on the bundle. Like it was the case for the instanton bundle, the $\O(S^{4})$-module of infinitesimal gauge transformations is generated by  
the ten generators of $so(5)$. The crucial difference (see Proposition \ref{prop:splitLieG-SO})
is that in the present example, vector fields of given degree $n$ in the generators of $\O(S^{4})$ split in the sum of two irreducible representations with distinct highest weight vectors with same weight.
 
We then start with the  $\O(S^4)$-module of equivariant derivations for $H:=\O(SO(4,\mathbb{R}))$.
 \beq\label{der-eq-S05}
{\rm{Der}}_{\M^H}(\O(SO(5, \IR)))= \{ X\in \Der{O(SO(5, \IR))} \, | \, \delta \circ X= (X \ot \id) \circ \delta \} \, .
\eeq
Let $\parn{IJ}$ denote the derivation on  $\O(SO(5))$ given on the generators by  $\parn{IJ}(n_{KL}) 
= \delta_{IK} \delta_{JL}$, for $I,J,K,L=1,\dots, 5$.  Then, recalling the coaction \eqref{coactionSO(4)},  ${\rm{Der}}_{\M^H}(\O(SO(5, \IR)))$
 is 
generated by the derivations
\begin{align}\label{der-sopra-SO}
&H_1:=  n_{1K} \parn{1K}- n_{3K} \parn{3K}
\quad  
&&H_2 := n_{2K} \parn{2K}- n_{4K} \parn{4K}
\nn
\\
& E_{10}   :=   n_{5K} \parn{3K}  - n_{1K} \parn{5K}
\quad  
&& E_{-10 } := n_{3K} \parn{5K} -n_{5K} \parn{1K}  
\nn
\\
& E_{01}  := n_{5K} \parn{4K} - n_{2K} \parn{5K}  
\quad  
&&E_{0 -1}  :=    n_{4K} \parn{5K} - n_{5K} \parn{2K}
\nn
\\
& E_{11}:=    n_{2K} \parn{3K} - n_{1K} \parn{4K} 
\quad  
&& E_{-1-1}   :=  n_{3K} \parn{2K} - n_{4K} \parn{1K}
\nn
\\
& E_{1-1}:= n_{4K} \parn{3K} - n_{1K} \parn{2K} 
\quad  
&& E_{-1 1}:=   n_{3K} \parn{4K} - n_{2K} \parn{1K}
\end{align}
with summation on  $K=1, \dots, 5$ understood. These ten generators close the Lie algebra of $so(5)$ in \eqref{so5}, from which the labels used. 
It is important to notice that due to the equivariance for the right coaction of $\O(SO(4,\mathbb{R}))$, 
 when applied to a generator $n_{JK}$, these derivations  do not move the second index. This fact will play a role later on.

Being equivariant, they restrict to derivations on the subalgebra $\O(S^4)$ of coinvariants:
\begin{align}\label{der-sotto-SO}
&H^\pi_1=  u_1 \partial_{u_1} - u_1^* \partial_{u_1^*} 
\quad  
&&H^\pi_2 = u_2 \partial_{u_2} - u_2^* \partial_{u_2^*} 
\nn
\\
& E^\pi_{10}   =   x  \partial_{u_1^*} - u_1 \partial_x
\quad  
&& E^\pi_{-10 } =   u_1^* \partial_x - x  \partial_{u_1} 
\nn
\\
& E^\pi_{01}  =  x  \partial_{u_2^*} - u_2 \partial_x
\quad  
&&E^\pi_{0 -1} =  u_2^* \partial_x - x  \partial_{u_2} 
\nn
\\
& E^\pi_{11}= u_2 \partial_{u_1^*} - u_1 \partial_{u_2^*} 
\quad  
&& E^\pi_{-1-1}   =  u_1^* \partial_{u_2}  - u_2^* \partial_{u_1} 
\nn
\\
& E^\pi_{1-1}= u_2^* \partial_{u_1^*} - u_1 \partial_{u_2} 
\quad  
&& E^\pi_{-1 1}=  u_1^* \partial_{u_2^*} - u_2 \partial_{u_1}  \; ,
\end{align}
using partial derivations for the generators $n_{j5}$ of $\O(S^4)$.
With the isomorphism in \eqref{iso4}   
these derivations coincide with the ones in \eqref{der-sotto}.  

The generic equivariant derivation is of the form 
$
X= b_1 H_1 + b_2 H_2 + \sum\nolimits_\mathsf{r}  b_\mathsf{r}  E_\mathsf{r} 
$, 
with  $b_j, b_\mathsf{r}  \in \O(S^4)$ and $H_j, E_{\mathsf{r}}$ in \eqref{der-sopra-SO}. 
The condition for $X$ to be vertical only uses its restriction to $\O(S^4)$, that is the derivations in \eqref{der-sotto-SO}
and thus coincides with the conditions \eqref{ker} under the isomorphism \eqref{iso4}.
Then, in parallel with Proposition \ref{3.1} for the generators \eqref{UWT2}, and noticing that its proof only uses the algebra structure of $\O(S^4)$, 
we have the following result  (we are dropping an overall factor of $2$). 

\begin{prop}\label{eqderort}
The Lie algebra  $\mathrm{aut}_{\O(S^4) }(\O(SO(5, \IR)))$ of infinitesimal gauge transformations of the   $\O(SO(4, \IR))$-Hopf--Galois extension 
 $\O(S^4 ) \subset \O(SO(5, \IR))$ is generated, as an $\O(S^4 )$-module, by the derivations 
\begin{align}\label{UWT2-SO}
  &K_1:= x H_2 + u_2^*  E_{01} + u_2  E_{0-1}  
\quad  &&
K_2:= x H_1 + u_1^*  E_{10}  + u_1  E_{-10} 
\nn \\
 & W_{01}:=   u_2 H_1 + u_1^*  E_{11} +  u_1 E_{-1 1}  
&&
  W_{0-1}:=     u_2^* H_1 +  u_1^* E_{1-1} + u_1 E_{-1-1}  
\nn \\
& W_{10}:=      u_1 H_2 - u_2^*   E_{11}  + u_2 E_{1-1} 
&&
  W_{-10}:=      u_1^*  H_2 + u_2^* E_{-11}  - u_2 E_{-1-1}  
\nn \\
 & W_{11}:=     x  E_{11} + u_1 E_{01}  - u_2 E_{10}  
&&
 W_{-1-1}:=   x E_{-1-1}  + u_1^* E_{0-1}  - u_2^* E_{-10}
\nn \\
& W_{1 -1}:=   -x E_{1-1}  + u_2^* E_{10} + u_1 E_{0-1}
&&
W_{-1 1}:=    -x E_{-11} + u_2 E_{-10} + u_1^* E_{01} . 
\end{align}
They are eigen-operators for $H_1$ and $H_2$ and transform under the adjoint representation [10] of $so(5)$ (that is \eqref{adjSO} hold), 
with $W_{11}$ the highest weight vector. 
\end{prop}
\noindent
In particular we have that $H_j \trl K_l =  0$ and $H_j \trl W_\mathsf{r} = r_j W_\mathsf{r}$, 
and this induces a (left, see later)  $\IZ^2$-grading 
on the derivations. They satisfy analogue relation of those in \eqref{relazioni-UWT2}: 
\begin{align}\label{relazioni-UWT2-so} 
&u_2   W_{0-1} - u_2^*   W_{01} + u_1  W_{-10} - u_1^*   W_{10} =0 
\nn
\\
& -   u_2  K_2 +  x   W_{01} - u_1^*  W_{11} + u_1  W_{-11}  =0 
\nn
\\
&  u_2^*  K_2 -     x W_{0-1} + u_1   W_{-1-1} - u_1^*  W_{1-1}  =0 
\nn
\\
&  -  u_1  K_1   +    x W_{10} + u_2^*  W_{11}  +u_2  W_{1-1}  =0 
\nn
\\
&    u_1^*  K_1  -     x W_{-10} - u_2  W_{-1-1} - u_2^*  W_{-11}=0 \; .
\end{align}
 
As in the case of the instanton bundle (cf. Lemma \ref{lem:5X10}), the fifty vector fields obtained by multiplying the vector fields in \eqref{UWT2-SO} with the generators of   $\O(S^4)$ can be arranged according to the representations 
$[35] \oplus [10] \oplus [5]$ of $so(5)$.  The highest weight vectors for these three representations are
given respectively by: 
\begin{align}\label{Yshat}
 Z_{21} &= u_1 W_{11} 
\, , \nn  \\
 Y_{11} &= x W_{11} + u_1 W_{01} - u_2 W_{10} \, ,  \nn \\
 X_{10} &= u_2^* W_{11} + u_2 W_{1-1} - u_1 K_1 + x W_{10}  ,
\end{align}
with the label denoting the value of the corresponding weight.  
These are the analogous of the vector fields found in \eqref{hwv-Y} for the $SU(2)$ Hopf bundle. 

When represented as vector fields on the bundle, the vector $X_{10}$ is zero  (because of relations analogous to \eqref{relazioni-UWT-2}) and  
the representation $[5]$ it generates  vanishes.    On the other hand, the vector field $ Y_{11}$ is no longer proportional to  $W_{11}$, as it was for the Hopf bundle
 case. 
An explicit computation, using also the condition $(N^\dag N)_{5K} = \delta_{5K}$ yields:
 \begin{align*}
W_{11} &= (x  n_{2K}  - u_2 n_{5K} ) \partial_{3K} +(- x n_{1K} +u_1 n_{5K}) \partial_{4K} +(u_2 n_{1K} - u_1 n_{2K}) \partial_{5K}
\\
{Y}_{11} &= 
u_1 (u_2  n_{1K} - u_1 n_{2K})   \partial_{1K} 
+ u_2( u_2 n_{1K} -u_1  n_{2K} )  \partial_{2K} 
\\ & \quad + (n_{2K} - u_2 \delta_{5K} +u_1^*  ( u_2 n_{1K} - u_1  n_{2K}  )) \partial_{3K} 
\\ & \quad +(-n_{1K} + u_1 \delta_{5K}+u_2^*  (u_2  n_{1K} - u_1 n_{2K})) \partial_{4K} 
+x( u_2 n_{1K} - u_1 n_{2K}) \partial_{5K} .
\end{align*}
The vector field ${Y}_{11}$ generates a different copy of the ten-dimensional representation of $so(5)$ that we denote by $\widehat{[10]}$ to distinguish it  from the ten-dimensional representation   $[10]$ of highest weight vector $W_{11}$.
Notice that while $[10]$ consists of vector fields which are combinations of those in \eqref{UWT2-SO} with coefficients of degree zero in the generators  of $\O(S^4)$, elements of 
$\widehat{[10]}$ are combinations with coefficients of degree one.

Next, in parallel with Lemma \ref{lem2.5}, the representation $[35] \oplus \widehat{[10]}$ of $so(5)$ just found are the ones that occur in the decomposition of the commutators.
\begin{lem}
The commutators of the derivations in  \eqref{UWT2-SO} can be organised according to
the representations $[35] \oplus \widehat{[10]}$ of $so(5)$ of highest weight vectors  $\alpha W_{11}$ and
$Y_{11}$ respectively. 
\begin{proof}
There are 45 non vanishing commutators. The non vanishing commutator with highest weight is 
$[W_{11}, W_{10}]$ with weight $(2,1)$. A direct computation shows that
$$
[W_{11}, W_{10}] = - u_1 W_{11}
$$
and the corresponding representation is the $[35]$ found above. 

The other highest weight vector, of weight $(1,1)$, is
$$
T_{11} = [K_{1}, W_{11}] + [K_{2}, W_{11}] + [W_{10}, W_{01}] =  4( x W_{11} - u_2 W_{10}    + u_1 W_{01}  ) = 4 Y_{11} \, .
$$
Thus the representation  generated by $T _{11}$ is the ten dimensional $\widehat{[10]}$.
\end{proof}
\end{lem}

By using the decomposition of $\O(S^4)$ in \eqref{splittingS4}, in contrast to Proposition \ref{prop:decoinst}, 
 we  have: 

\begin{prop}\label{prop:splitLieG-SO}
There is a decomposition
$$
\mathrm{aut}_{\O(S^4) }(\O(SO(5, \IR)))=\bigoplus_{n\in \IN_0} \, [d(2,n)] \oplus \widehat{[d(2,n-1)]}~.
$$
Here $[d(2,n)]$, respectively $\widehat{[d(2,n-1)]}$, is the representation of $so(5)$ with highest weight vector $\alpha^n W_{11}$ of weight $(n+1,1)$, respectively $\alpha^{n-1} Y_{11}$ of weight $(n,1)$; they consist of  derivations on $\O(SO(5, \IR))$ which are combinations of the derivations in \eqref{UWT2-SO} with polynomials coefficients of degree $n$   in the generators of $\O(S^4)$.  
\end{prop}
 
\subsubsection{Braided derivations and infinitesimal gauge transformations}\label{sec:so4-twist}

Let us now pass to the twisted Hopf--Galois extension $\O(S^{4}_\theta)=\O(SO_\theta(5,\mathbb{R}))^{co \O(SO_\theta(4,\mathbb{R}))} \subset
\O(SO_\theta(5,\mathbb{R}))$. 

We briefly recall its construction from twist deformation (see  \cite[\S 4.1]{ppca} for details). 
Consider the 2-cocycle 
$\gamma: \O(\mathbb{T}^2)\otimes
\O(\mathbb{T}^2)\to \mathbb{C}$ on $\O(\mathbb{T}^2) \subset \O(SO(4, \IR))$, 
given on the generators by
$\gamma(t_1\otimes t_2)=e^{-\pi i \theta}$,
$\gamma(t_2\otimes t_1)=e^{ \pi i \theta}$ and
${\gamma(t_1\otimes t_1)=\gamma(t_2\otimes t_2)=1}$.
Notice that here $\gamma=  \sigma^2$, where $\sigma$ is the cocycle used in \S \ref{sec:braidedHopf}.
We use it to deform   the Hopf
algebra $\O(SO(4, \IR))$ into the noncommutative Hopf algebra
$\O(SO_\theta(4, \IR))$. This latter has same coalgebra structure as the original one but twisted algebra multiplication,
$$
m_{IJ} \dott m_{KL}= \co{T_I}{T_K} m_{IJ} m_{KL}\coin{T_J}{T_L} , \quad
I,J,K,L=1, \dots, 4, 
$$
where  $T := \mathrm{diag}(t_1, t_2, t_1^*,  t_2^*)$. 
(Here again, to conform with the literature we use the subscript $\theta$ 
instead of $\gamma$ for twisted algebras and their multiplications.)
We set  $\defp$ 
$$
\defp_{IJ}:=(\gamma(T_I\otimes T_J))^2
$$ 
so that
$\defp_{I J}= \exp(-2i\pi \theta_{I J})$. Since
$\bar\gamma(T_J\otimes T_L)=\gamma(T_L\otimes T_J)$,
and $\gamma(T_L\otimes T^*_J)=\bar\gamma(T_L\otimes T_J)$ we
  have $\defp_{JI}={\defp_{IJ}}^{-1}=\defp_{IJ'}$.
 It follows that the generators in $\O(SO_\theta(4, \IR))$ satisfy the commutation relations
\beq\label{comM}
m_{IJ} \dott m_{KL} = \defp_{IK}  \defp_{LJ} \, m_{KL} \dott m_{IJ},\quad I,J,K,L=1, \dots, 4 .
\eeq

The twisted antipode turns out to be equal to the starting one,   $S_\theta (m_{IJ}) =S(m_{IJ}).$
The relations  (\ref{idealQ}) become  
\begin{eqnarray}
\label{qidealQ}
M^t \dott Q \dott M =Q~,~~
 M \dott Q \dott M^t=Q~,~~
\end{eqnarray}
together with ${\det}_{\theta}(M)=1$.

Next, using the projection $\pi$ in \eqref{pi} we lift the 
 2-cocycle from the subgroup $\O(SO(4, \IR))$
to the Hopf algebra $\O(SO(5, \IR))$ (or equivalently we can consider the same torus
$\mathbb{T}^2$ embedded in $SO(5)$). The resulting Hopf algebra is
denoted by $\O(SO_\theta(5, \IR))$ and has 
generators    $n_{IJ}$ with relations 
\begin{equation}\label{comN}
n_{IJ} \dott n_{KL}= \defp_{IK}\defp_{LJ}  n_{KL} \dott n_{IJ},\quad I,J,K,L=1, \dots, 5 ,
\end{equation}
where now $T := \mathrm{diag}(t_1, t_2,
t_1^* ,  t_2^*,1),$ and orthogonality conditions
$N^t \dott Q \dott N =Q$ and $N \dott Q \dott N^t=Q$, with ${\det}_{\theta}(N)=1$.

The quantum homogeneous space $\O(S^{4})$ is deformed into the
quantum homogeneous space $\O(S_\theta^{4})\subset \O(SO_\theta(5, \IR))$, consisting of coinvariants of $\O(SO_\theta(5, \IR))$ 
under the $\O(SO_\theta(4, \IR))$-coaction. The algebra $\O(S_\theta^{4})$
  is
generated by five elements $\{u_I\}_{I=1,..., 5}=\{u_i, u_{i'}=u^*_i, x\}_{i=1,2}$ with  commutation relations, obtained from \eqref{comN},
\beq
u_I\dott u_J=\defp_{IJ}u_J\dott u_I\label{crqos}
\eeq
The orthogonality conditions   imply the sphere
relation 2 $\sum\nolimits_{i=1}^{2}  u_i^* \dott u_i +x^2=1$. Moreover, 
from the general theory developed in \cite{ppca}, 
the algebra extension $\O(S_\theta^{4})\subset \O(SO_\theta(5, \IR))$ 
of the quantum homogeneous space $\O(S_\theta^{4})= \O(SO_\theta(5, \IR))^{co\O(SO_\theta(4, \IR))}$ is still  Hopf--Galois.

When considering the braided Lie algebra of infinitesimal gauge
transformations of this Hopf--Galois extension, it is useful to think
of the latter as the result of a double deformation done with
commuting left coaction of $\O(\mathbb{T}^2)$ and right coaction of $\O(\mathbb{T}^2) \subset \O(SO(4, \IR))$.  
This second $\O(\mathbb{T}^2)$   disappears when considering $\O(SO(4, \IR))$ equivariant quantities. This is the case for the algebra  
$\O(S^{4})$ of coinvariant elements. It is also the case for the equivariant derivations in Proposition \ref{eqderort} and it is in this sense that those derivations can be thought of as having trivial right $\IZ^2$-grading (they do not move the second index in a generator $n_{JK}$ as already mentioned). 

Thus for the braided Lie algebra of infinitesimal gauge transformations of the Hopf--Galois extension 
$\O(S_\theta^{4})= \O(SO_\theta(5, \IR))^{co\O(SO_\theta(4, \IR))}$ we just need to consider 
the left torus action and the construction goes exactly as for the $SU(2)$ instanton case of the previous section. In particular we can repeat the construction 
in \S \ref{sec:braidedHopf} verbatim by considering the maximal torus
 $\mathbb{T}^2 \subset SO(5)$, generated by the right
 invariant vector fields
 $H_1$ and $H_2$ of $SO(5)$, and use  
the twist
\beq\label{F-so}
\F:= e^{\pi i\theta (H_1 \ot H_2 -H_2 \ot H_1)} \in K\otimes K\subset
U(so(5))^{op}\otimes U(so(5))^{op}\; , \quad \theta \in \IR ,
\eeq
of $K$, where $K$ is generated by the right invariant vector fields
$H_1$ and $H_2$. Hence $K$ is the universal enveloping algebra of the Cartan subalgebra of $so(5)^{op}$, the
Lie algebra $so(5)$ being that of left invariant vector fields on
$SO(5)$, cf. \cite[\S 7.1]{pgc-main}.
The braided Lie algebras of  braided derivations and of infinitesimal gauge transformations of the Hopf--Galois extension 
 $\O(S^4_\theta) \subset \O(SO_\theta(5, \IR))$ are the twist (left)  deformations of 
 $ {\rm{Der}}_{\M^H}(\O(SO(5, \IR)))$ and of 
 $\mathrm{aut}_{\O(S^4) }(\O(SO(5, \IR)))$ respectively (with the right torus action playing no role).

As  $\O(S^4)$-modules, 
the Lie algebra ${\rm{Der}}_{\M^H}(\O(SO(5, \IR)))$, $H=\O(SO(4, \IR))$, is generated by the operators 
in \eqref{der-sopra-SO} while the Lie algebra of infinitesimal gauge transformations 
 $\mathrm{aut}_{\O(S^4)}(\O(SO(5, \IR)))$  is generated by the operators  in \eqref{UWT2-SO}. 

From Corollary \ref{DAmodiso} we have that
${\rm{Der}}_{\M^H}(\O(SO_\theta(5, \IR))) = \dd(  {\rm{Der}}_{\M^H}(\O(SO(5, \IR) )_\F)$ for 
$H=O(SO_\theta(4, \IR))$. Thus: in parallel with Proposition
\ref{derivationsinstanton},  we have
\begin{prop}\label{der-hom}
The braided Lie algebra ${\rm{Der}}_{\M^H}(SO_\theta(5, \IR))$ of equivariant derivations 
of the  $\O(SO_\theta(4, \IR))$-Hopf--Galois extension  $\O(S^4_\theta) \subset \O(SO_\theta(5, \IR))$
is generated, as an  $\O(S^4_\theta)$-module, by elements
\begin{equation}\label{DgenerDer-SO}
\widetilde{H}_j :=\dd (H_j) \, , \quad  \widetilde{E}_\mathsf{r} := \dd(E_\mathsf{r})\;
, \quad j=1,2 \, , \quad \mathsf{r} \in \Gamma
\end{equation}
with bracket 
\begin{align*}
&[\widetilde{H}_1,\widetilde{H}_2]_\rF =\dd([{H}_1,{H}_2]) = 0 \; ; \qquad
&&
[\widetilde{H}_j,\widetilde{E}_\mathsf{r}]_\rF = \dd([ {H}_j,  {E}_\mathsf{r}]) 
= r_j \widetilde{E}_\mathsf{r}\; ; 
\\
&[\widetilde{E}_\mathsf{r}, \widetilde{E}_{-\mathsf{r}}]_\rF = \dd([ {E}_\mathsf{r},  {E}_{-\mathsf{r}}])
= \sum\nolimits_j r_j \widetilde{H}_j   \; ;
&&
[\widetilde{E}_\mathsf{r}, \widetilde{E}_\mathsf{s}]_\rF = e^{- i \pi \theta {\mathsf{r}} \wedge {\mathsf{s}} } \, 
\dd([ {E}_\mathsf{r},  {E}_\mathsf{s}]) = e^{- i \pi \theta {\mathsf{r}} \wedge {\mathsf{s}} } N_\textsf{rs} \widetilde{E}_\textsf{r+s}\; 
\end{align*}
with $N_\textsf{rs}=0$ if $\textsf{r+s}$ is not a root. 
\end{prop}
Similarly, Proposition \ref{autautF}
yields $\mathrm{aut}_{\O(S^4_\theta) }(\O(SO_\theta(5, \IR)) )
=\dd(\mathrm{aut}_{\O(S^4_\theta) }(\O(SO(5, \IR))_\theta)$.
Thus, in parallel with Proposition \ref{gaugeinstanton}:
\begin{prop}\label{derver-hom}
The braided Lie algebra $\mathrm{aut}_{\O(S^4_\theta) }(\O(SO_\theta(5, \IR)) )$ 
of infinitesimal gauge transformations of the $\O(SO_\theta (4, \IR))$-Hopf--Galois extension 
$\O(S^4_\theta) \subset \O(SO_\theta(5, \IR))$
is generated, as an  $\O(S^4_\theta)$-module, by the elements
\begin{equation}\label{Dgenerators-SO}
\widetilde{K}_j :=\dd (K_j) \, , \quad  \widetilde{W}_\mathsf{r} := \dd(W_\mathsf{r})\;
, \quad j=1,2 \, , \quad \mathsf{r} \in \Gamma
\end{equation}
with bracket 
\begin{align*}
[\widetilde{K}_1,\widetilde{K}_2]_\rF &=\dd([{K}_1,{K}_2]) \; ; \quad
[\widetilde{K}_j,\widetilde{W}_\mathsf{r}]_\rF = \dd([{K}_j,{W}_\mathsf{r}]) \; ; \\
[\widetilde{W}_\mathsf{r}, \widetilde{W}_\mathsf{s}]_\rF &= e^{- i \pi \theta {\mathsf{r}} \wedge {\mathsf{s}} } \, \dd([{W}_\mathsf{r}, {W}_\mathsf{s}]) \; .
\end{align*}
The braided Lie bracket of generic elements in $\mathrm{aut}_{\O(S^4_\theta)
}(\O(SO_\theta(5, \IR)) )$ is given by
\begin{equation}\label{Lierelations1}
    [b \dott \widetilde{X}, b'\dott \widetilde{X}']_{\r_\F} = b\dott (\rF_\alpha\trl b')  \dott [\rF^\alpha\trl
    \widetilde{X},\widetilde{X}']_{\r_\F}~   
  \end{equation}
for $b,b'\in \O(S^4_\theta)$ and $\widetilde{X},\widetilde{X}'$ in the linear span of the generators in \eqref{Dgenerators-SO}.
\end{prop}

\appendix
\section{Decomposition of  $\O(S^4)$}\label{app:decS4} 
We thus study how $\O(S^4)$  decomposes in the sum of 
irreducible representations $[d(s,n)]$ of $so(5)$. 
In the algebra $\O(\IR^5)$, both  $\alpha$ and  $\rho^2=\alpha\alpha^*+\beta \beta^*+x^2$ are annihilated by all 
raising operators $W_\mathsf{r}$  (the ones for positive roots), and thus their powers and products. 
They are of weight $(1,0)$ and $(0,0)$ respectively.

Let $V^{\scriptscriptstyle{(r)}}$ be the $\binom{4+r}{r}$-dimensional vector space 
of monomials of degree $r$ in the indeterminates $\alpha,\alpha^*,\beta,\beta^*,x$.   
The  vectors $\alpha^{r-2k}\rho^{2k}$ are highest weight vectors of $V^{\scriptscriptstyle{(r)}}$ and
\beq\label{splittingR5}
 V^{\scriptscriptstyle{(r)}} =  \bigoplus_{k=0}^{\lfloor \tfrac{r}{2}\rfloor} [d(0,r-2k)] 
\eeq
where $[d(0,r-2k)]$ is the irreducible representation  with highest weight vector $\alpha^{r-2k}\rho^{2k}$ of weight $(r-2k,0)$, 
and of dimension  $d(0,r-2k)=\tfrac{1}{6} (1+r-2k)(2+r-2k)(3+2r-4k)$. Indeed, 
for $r=2 m$, 
\begin{align*}
\sum\nolimits_{k=0}^m d(0,2m-2k)&= \sum\nolimits_{k=0}^m d(0,2m) = \tfrac{1}{3} \sum\nolimits_{k=0}^m (3+ 13 k + 18 k^2 +8 k^3)
\\
&=
\tfrac{1}{3} \left(
3 + 13 \, \tfrac{m(m+1)}{2}  + 18 \, \tfrac{m(m+1)(2m+1)}{6}  + 8  \, \tfrac{m^2(m+1)^2}{4}  
\right) 
\\
&=
\tfrac{1}{6} (m+1)(6 + 19 m +16 m^2 + 4 m^3) \, ,
\end{align*}
which coincides with the dimension
$ \binom{4+2m}{2m}$ of $V^{\scriptscriptstyle{(2m)}}$. Similar computations go for $r$ odd.

For $\alpha,\alpha^*,\beta,\beta^*,x$ coordinate functions on $\O(S^4)$, $\rho^2=1$ and with fixed  $r$, all representations $[d]$ in the decomposition \eqref{splittingR5} already appeared in $V^{\scriptscriptstyle{(r')}} $  for some $r' <r$, but for $[d(0,r)]$.
We hence conclude that
$$
 \O(S^4)  =  \bigoplus_{n \in \IN_0} [d(0,n)] 
$$
where $[d(0,n)]$ has highest weight vector $\alpha^{n}$ of weight $(n,0)$.  

\section{Matrix representation of the braided Lie algebra $so_\theta(5)$}\label{app:mat-rep-so5}

We give a matrix representation of the braided Lie algebra $so_\theta(5)$ as defined in \eqref{sotheta5}. 

Consider weights $\mu, \nu= (0,0), (\pm 1,0), (0,\pm 1)$ 
with order 
$$
(1,0)= 1 \, , \;  (0,1) = 2 \, , \;  (-1,0) = 3 \, , \;  (0,-1) = 4 \, , \;  (0,0) = 5.
$$ 
Using this order for an identification between weights and row/column indices, define matrices 
$\el_{\mu \nu}$ of components 
$$
(\el_{\mu \nu})_{\sigma \tau} := \lambda^{\mu \wedge \nu} \delta_{\mu \sigma} \delta_ {\nu \tau} .
$$
The product of  two such matrices is found to be 
$$
\el_{\mu \nu} \el_{\tau \sigma} = \lambda^{  (\mu - \nu) \wedge ( \tau - \sigma )} \delta_{\nu \tau} \el_{\mu \sigma} .
$$
The minus signs in the exponents are due to $\el_{\mu \nu}$ having weight $\mu-\nu$ (cf.~\eqref{comM}).

A direct computation shows the following:
 \begin{lem}\label{lemmab}
 The matrices
 \begin{align*}
 & \wee_1 := \el_{11}- \el_{33}
 &&  
\wee_2  :=   \el_{22}- \el_{44}
\nn
\\
&  \wee_{10}  :=  \el_{15}- \el_{53}
&&  
 \wee_{-10}  :=   \el_{51}- \el_{35}
\nn
\\
&  \wee_{11} :=  \el_{14}- \el_{23}
&&  
 \wee_{-1-1} :=  \el_{41}- \el_{32}
\nn
\\
&  \wee_{01}  :=   \el_{25}- \el_{54}
&&  
 \wee_{0-1}  :=   \el_{52}- \el_{45}
\nn
\\
&  \wee_{1-1}  :=  \el_{12}- \el_{43}
&&  
 \wee_{-11}  :=  \el_{21}- \el_{44}
 \end{align*}
 give a matrix representation of the algebra $so_\theta(5)$ in \eqref{sotheta5} with the identification $\wee_\mathsf{r} \leftrightarrow \widetilde{E}_\mathsf{r}$ and setting $[\wee_\mathsf{r}, \wee_\mathsf{s}]_\rF := \wee_\mathsf{r}\wee_\mathsf{s} - \lambda^{2 \mathsf{r} \wedge \mathsf{s}} \wee_\mathsf{s}\wee_\mathsf{r}$ for the braided commutator of matrices.
 \end{lem}
The matrices  $\wee_\mathsf{r}$ in the lemma are of the form
\beq
\wee_{\mu, \nu}:= \el_{\mu \nu^*} -   \el_{\nu \mu^*} \; , \qquad \mbox{for} \quad \mu +\nu=\mathsf{r} \, .
\eeq
The braided commutator $[\wee_{\mu, \nu}, \wee_{\tau, \sigma}]_{\r_\F} =\wee_{\mu, \nu} \wee_{\tau, \sigma}   - \lambda^{2 (\mu + \nu) \wedge (\tau+ \sigma))}\wee_{\tau, \sigma}  \wee_{\mu, \nu} 
$ 
is found to be
\beq\label{so5-rep}
 [\wee_{\mu, \nu}, \wee_{\tau, \sigma}]_{\r_\F} 
= \lambda^{  (\mu +\nu)\wedge ( \tau+\sigma )}  \big(
   \delta_{\nu^* \tau }   \wee_{\mu, \sigma} 
-  \delta_{\mu^* \tau }   \wee_{\nu, \sigma}
  -   \delta_{\nu^* \sigma }     \wee_{\mu, \tau}  
  +     \delta_{\sigma^* \mu }      \wee_{\nu, \tau}  
  \big).
 \eeq
In the classical limit, $\lambda=1$, the matrices $\el_{\mu \nu}$ reduce to the usual elementary matrices and those in 
Lemma \ref{lemmab} give the defining matrix representation of the Lie algebra $so(5)$.

\section{The Lie algebra  $\mathrm{aut}_{\O(S^4)}(\O(S^7))$}\label{app:commut}
In the following table we list the action of   the generators \eqref{UWT2} of
  the Lie algebra of infinitesimal gauge transformations  on the generators of the algebra
$\O(S^7)$.
\begin{table}[H]\caption{Vertical vector fields on $\O(S^7)$}
  \label{table:vfa}
\bigskip
\begin{center}
\scalebox{0.9}{
\begin{tabular}{c||c|c|c|c}
 & $ z_1 $ & $ z_2 $ & $ z_3 $ & $ z_4    $
\\ \hline \rule[-4mm]{0mm}{1cm}
$K_1$ & $- x z_1 + \beta^*   z_4$ & $ x  z_2 + \beta z_3 $ & $ - x z_3 + \beta^*   z_2  $ & $ x z_4 +  \beta z_1  $
\\ \hline  \rule[-4mm]{0mm}{1cm}
$K_2 $ & $  x z_1 +  \alpha  z_3$&$ - x z_2  -  \alpha^*  z_4 $ & $ - x z_3 + \alpha^*   z_1 $&$ x z_4 -    \alpha  z_2$
\\ \hline  \rule[-4mm]{0mm}{1cm}
$ W_{01} $&$  \stwo \beta z_1  -  \sqrt{2} \alpha z_2  $&$ - \stwo \beta z_2  $&$ - \stwo  \beta z_3  -  \sqrt{2} \alpha^*   z_4 $&$ \stwo \beta z_4 $
\\ \hline   \rule[-4mm]{0mm}{1cm}
$W_{0 -1} $&$  \stwo  \beta^* z_1 $&$ - \stwo  \beta^* z_2  - \sqrt{2} \alpha^*  z_1 $&$ - \stwo  \beta^* z_3 $&$ \stwo  \beta^* z_4  - \sqrt{2} \alpha z_3$
\\ \hline  \rule[-4mm]{0mm}{1cm}
$ W_{10} $&$   - \stwo \alpha z_1 $&$  \stwo \alpha z_2 - \sqrt{2} \beta z_1  $&$ - \stwo \alpha z_3 + \sqrt{2} \beta^*  z_4 $&$ \stwo \alpha z_4   $
\\ \hline  \rule[-4mm]{0mm}{1cm}
$W_{-10} $&$   - \stwo \alpha^*  z_1  - \sqrt{2} \beta^* z_2  $&$  \stwo \alpha^*  z_2  $&$ - \stwo \alpha^*  z_3 $&$ \stwo \alpha^*  z_4 + \sqrt{2} \beta z_3 $
\\ \hline  \rule[-4mm]{0mm}{1cm}
$ W_{11}$ & $    \alpha  z_4 $ & $ \beta z_4  $&$ - 2x z_4 + \alpha z_2  -\beta z_1 $&$0  $
\\ 
&&& $= (1-x) z_4$ &
\\ \hline   \rule[-4mm]{0mm}{1cm}
$W_{-1-1}  $&$ - \beta^* z_3   $&$ \alpha^* z_3 $&$ 0 $&$ - 2x z_3 + \alpha^* z_1 + \beta^* z_2$
\\
&&&&=$ (1-x) z_3$
\\ \hline  \rule[-4mm]{0mm}{1cm}
 $W_{1-1}$&$    0 $&$ 2x  z_1 - \beta^*  z_4 +\alpha  z_3 $&$  \beta^*  z_1   $&$ \alpha  z_1$
\\
&& $= (1+x) z_1  $ &&
\\ \hline    \rule[-4mm]{0mm}{1cm}
$ W_{-1 1}$&$  2x  z_2  + \beta z_3  +\alpha^*    z_4 $&$ 0 $&$  \alpha^*  z_2  $&$- \beta z_2$
\\
& $= (1+x) z_2$&&&
\\
\end{tabular} 
}
\end{center}
\end{table}

We list the commutators of the generators \eqref{UWT2} of the Lie algebra of infinitesimal gauge transformations 
 $\mathrm{aut}_{\O(S^4)}(\O(S^7))$, obtained by direct computation. 
 \begin{align*}
& [K_1, K_2] = \sqrt{2} ( \alpha^*  W_{10} - \alpha W_{-10} )
 \\
&[K_1, W_{01}] = - \sqrt{2} \beta K_2 + 2 x    W_{01} 
&&[K_2, W_{01}] = \sqrt{2} ( \alpha^*  W_{11} + \alpha W_{-11} ) 
 \\
&[K_1, W_{1-1}] =- 2 x W_{1-1} + \sqrt{2} \beta^* W_{10}
&&[K_2, W_{1-1}]=2 x W_{1-1} - \sqrt{2} \alpha W_{0-1} 
 \\
&[K_1, W_{10}] =\sqrt{2} (- \beta^*  W_{11} +\beta W_{1-1} )
&&[K_2, W_{10}]= 2x W_{10} - \sqrt{2}  \alpha  K_1
 \\
&[K_1, W_{11}] = 2 x W_{11} - \sqrt{2} \beta W_{10}
&&[K_2, W_{11}]=2 x W_{11} + \sqrt{2} \alpha W_{01}   
 \end{align*}
 \begin{align*}
 & 
 [W_{01}, W_{1-1}]= \sqrt{2} \beta W_{1-1} + \sqrt{2} \alpha (-K_1 +K_2 )
&& 
[W_{-1-1}, W_{01}]=     \sqrt{2} \beta W_{-1-1} - \sqrt{2} \alpha^* (K_1 +K_2 )
\\
&
[W_{01}, W_{10}]=  \sqrt{2} (\beta W_{10} - \alpha W_{01} )  
&&
 [W_{-1-1}, W_{1-1}] =  \sqrt{2}   \beta^* W_{0-1} 
\\
&
 [W_{01}, W_{11}]=  \sqrt{2} \beta W_{11}  
 &&
[W_{-1-1} , W_{10}]=  \sqrt{2}   \alpha W_{-1-1}  + \sqrt{2} \beta^* (K_1 +K_2 )
\\
&
[W_{1-1} ,W_{10}]=     \sqrt{2}   \alpha W_{1-1} 
 \end{align*}

 \begin{align*}
  & \,
&&
[W_{-1-1} ,W_{11}]= -2x (K_1 \!+K_2 ) -\! \sqrt{2} (  \alpha W_{-10}  \!+ \!  \beta W_{0-1} )
\\
&
[W_{1-1}, W_{11}]=  - \sqrt{2}   \alpha W_{10} 
&&
[W_{-10}, W_{01}]=   \sqrt{2} (\beta W_{-10} + \alpha^* W_{01} )\!\!
\\
&
[W_{10}, W_{11}]=   \sqrt{2}   \alpha W_{11} 
&&
[W_{-10}, W_{1-1} ]=  -  \sqrt{2}   \alpha^* W_{1-1}  +  \sqrt{2} \beta^* (- K_1 +K_2 )
\\
  &
[W_{10}, W_{-10}]= \sqrt{2} (\beta^* W_{01} \!+\! \beta W_{0-1} )
&&
[W_{0-1}, W_{01}]=  \sqrt{2} ( \alpha^*   W_{10} + \alpha W_{-1 0} )
\\
&
[W_{-11},W_{01}]=  \sqrt{2} \beta W_{-11}
&&
[W_{-11}, W_{1-1}] = 2x ( K_1 \!-\! K_2 ) - \sqrt{2}   (\beta^* W_{01}  -   \alpha W_{-10}) 
 \end{align*}

The remaining brackets are obtained  using the $*$-structure:
$$
[K_j, W_{-\mathsf{r}}] 
= [K_j^*, W_\mathsf{r}^*] 
= - ([K_j, W_\mathsf{r}])^* \; , 
\qquad
[W_{- \mathsf{r}} , W_{-\mathsf{s}}] 
= [W_\mathsf{r}^*, W_\mathsf{s}^*] 
=-([W_\mathsf{r} , W_\mathsf{s}])^* \, ,
$$
with $K_j^*=K_j$ and $W^*_\mathsf{r}=  W_{-\mathsf{r}}$,  see page \pageref{UWT2}, and
  $(bX)^*= b^* X^*$ for each $b \in \O(S^4)$ and $X$ a derivation. For example, one computes
\begin{align*}
& [K_2, W_{0-1}]
 = -( [K_2, W_{01}])^*
 = -  \sqrt{2} ( \alpha^*  W_{1-1} + \alpha W_{-1-1} )
 \\
 &
 [W_{10} , W_{0-1}]
 = - ([W_{-10}, W_{01}])^*
 = -  \sqrt{2} (\beta^* W_{10} + \alpha W_{0-1} )
 \; .
 \end{align*}
\\[2em]
\noindent
\textbf{Acknowledgments.}~\\[.5em]
This research has a financial support from Universit\`a del Piemonte Orientale. 
PA acknowledges  partial support from INFN, CSN4, Iniziativa
Specifica GSS and from INdAM-GNFM. PA acknowledges hospitality
  from Universit\`a di Trieste.
GL acknowledges partial support from INFN, Iniziativa Specifica GAST
and INFN Torino DGR4.
CP acknowledges support from Universit\`a di Trieste (assegni
Dipartimenti di Eccellenza,  legge n. 232 del 2016).  
GL and CP acknowledge partial support  from the ``National Group for Algebraic and Geometric Structures, and their
Applications'' (GNSAGA--INdAM) and hospitality from Universit\`a del  Piemonte
Orientale and INFN Torino.

\end{document}